\documentclass[a4paper,10pt]{article}

\usepackage{graphicx}
\usepackage{amsmath,amssymb,amsfonts,amsthm}
\usepackage{url}
\usepackage{hyperref}
\usepackage{color}
\usepackage{ifthen}
\usepackage{subcaption}
\usepackage{mathtools}
\usepackage{lipsum}
\usepackage{authblk}

\newcommand{\R}{\mathbb{R}}
\newcommand{\N}{\mathbb{N}}

\newcommand{\inner}[2]{\ifthenelse{\equal{#2}{}}{\left\langle\cdot,\cdot\right\rangle_{#1}}{\left\langle#2\right\rangle_{#1}}}
\newcommand{\norm}[2]{\ifthenelse{\equal{#2}{}}{\left\|\cdot\right\|_{#1}}{\left\|#2\right\|_{#1}}}
\newcommand{\seminorm}[2]{\ifthenelse{\equal{#2}{}}{\left|\cdot\right|_{#1}}{\left|#2\right|_{#1}}}

\newcommand{\fa}{\hbox{ for all }}

\newtheorem{theorem}{Theorem}
\newtheorem{proposition}[theorem]{Proposition}
\newtheorem{corollary}[theorem]{Corollary}
\newtheorem{lemma}[theorem]{Lemma}
\newtheorem{definition}[theorem]{Definition} 
\newtheorem{remark}[theorem]{Remark}

\DeclareMathOperator{\cssim}{cSSIM}

\DeclareMathOperator{\sinc}{sinc}
\newcommand{\bs}{\boldsymbol}
\DeclareMathOperator{\wssim}{W-SSIM}
\DeclareMathOperator{\wcssim}{W-cSSIM}

\title{Convergence results in image interpolation with the continuous SSIM}

\author[1]{Francesco Marchetti\thanks{\url{francesco.marchetti@unipd.it}}}
\author[2]{Gabriele Santin\thanks{\url{gsantin@fbk.eu}}}

\affil[1]{Department of Mathematics \lq\lq Tullio Levi-Civita", University of Padova, Italy}
\affil[2]{Digital Society Center, Bruno Kessler Foundation (FBK), Trento, Italy}

\begin{document}

\maketitle

\begin{abstract}
Assessing the similarity of two images is a complex task that attracts significant efforts in the image processing community.
The widely used Structural Similarity Index Measure (SSIM) addresses this problem by quantifying a perceptual structural similarity.
In this paper we consider a recently introduced continuous SSIM (cSSIM), which allows one to analyze sequences of images of increasingly fine resolutions, and further extend the definition of the index to encompass the locally weighted version that is used in practice. 
For both the local and the global versions, we prove that the continuous index includes the classical SSIM as a special case, and we provide a precise connection between image similarity measured by the cSSIM and by the $L_2$ norm.
Using this connection, we derive bounds on the cSSIM by means of bounds on the $L_2$ error, and we even prove that the two error measures are equivalent in certain circumstances. We exploit these results to obtain precise rates of convergence with respect to the cSSIM for several concrete image interpolation methods, and we further validate these findings by different numerical experiments.
This newly established connection paves the way to obtain novel insights into the features and limitations of the SSIM, including on the effect of the local weighted window on the index performances. 
\end{abstract}



\section{Introduction}

Image interpolation is a widely investigated topic in image processing and concerns various applications, such as super-resolution, image resizing, image rotation and registration. 

Many different interpolation techniques have been developed in the last decades, and are used in different settings. Nearest-neighbor, bilinear and bicubic interpolation are classical fast polynomial-based approaches \cite{Jiang15,Rukundo14,Smith81}, which however present limitations especially in the presence of edges \cite{Siu12}. In order to treat the resulting artifacts, adaptive techniques have been proposed \cite{AlNasrawi17,Hwang04,Khan19}.
In certain applications, an image needs instead to be reconstructed by interpolating at non gridded data sites. In this case, unless it is possible to exploit some particular properties concerning the distribution of the samples \cite{DeMarchi17}, a kernel-based approach guarantees the flexibility required when dealing with scattered data \cite{DeMarchi20}.
Furthermore, various deep learning strategies based on Convolutional Neural Networks (CNNs) have been designed in the last years and turned out to be accurate and efficient tools in many tasks \cite{Chen18,Huang17,Yamanaka17,Zeng19}.

Independently of the specific method of choice, evaluating the adherence of the final reconstruction to the original image is itself a challenging task, and classical error metrics have been found in many cases to be unsuitable for this goal.
The Structural Similarity Index Measure (SSIM) \cite{Wang04} is a extensively used metric 
that aims at quantifying the \emph{visually perceived} quality of the reconstruction. In the last years its mathematical properties have been thoroughly investigated, and different result have been derived \cite{Brunet12,Brunet12a}. Moreover, modifications of the SSIM has been proposed, e.g. in \cite{Sampat09,Rehman15}. In particular, the continuous SSIM (cSSIM) has been introduced in \cite{Marchetti2021a} as the extension of the \textit{discrete} SSIM to a continuous framework. There, the convergence rate of a kernel-based interpolant in terms of the cSSIM has been analysed in terms of the supremum norm.

In this work, we consider again the cSSIM and start by extending it to a local, weighted index, namely the W-cSSIM, inspired from the classical discrete counterpart. Then, we provide a clean theoretical link between these continuous indices and their discrete versions, thus improving what outlined in \cite{Marchetti2021a}, showing that the cSSIM (W-cSSIM) is indeed a limit version of the SSIM (W-SSIM) as the resolution of the images gets larger. Then, in Section \ref{sec:cssim_l2} we analyse both the cSSIM and the W-cSSIM and prove that the dissimilarity index $1-\mathrm{cSSIM}$ can be bounded by the square of the $L_2$ norm. We point out that our analysis differs significantly from the one carried out in \cite{Brunet12a}, where a normalized SSIM-derived $L_2$ metric is compared to the SSIM through a statistical approach. While sharing a similar spirit, by virtue of the cSSIM here the interlacing between SSIM and $L_2$ norm is interpreted and formalized from a different perspective.

In Section \ref{sec:equivalence} we prove that, under certain assumptions, even lower bounds on the dissimilarity index in terms of the $L_2$ norm can be derived, and we thus show that, in this case, the dissimilarity index $1-\mathrm{cSSIM}$ is in fact \textit{equivalent} to the $L_2$ norm. A discussion of these assumptions sheds some light on the opposite situation, when indeed the SSIM is able to detect similarities that are not measured by the $L_2$ error, as it is often observed in practice (see e.g. \cite{Wang2009}).

These theoretical findings are exploited in Section \ref{sec:int_methods} to formulate accurate convergence rate estimates for bilinear, Hermite bicubic and kernel-based interpolation, which are confirmed by the numerical tests carried out in Section \ref{sec:num_tests}. Then, in Section \ref{sec:img_interp} we perform some concrete image interpolation experiments, showing how the limitations in regularity affect the theoretically achieved convergence rates. The discussion of the obtained results, as well as some concluding remarks, are offered in Section \ref{sec:conclusions}.

\section{The Structural Similarity Index and its continuous counterpart}

We start by recalling the definitions of SSIM and cSSIM. Letting $F,G\in\R_{\ge 0}^{p\times q},\;p,q\in\mathbb{N}$, be positive-valued matrices representing 
two single-channel images, the SSIM, as introduced in the original paper \cite{Wang04}, is defined as
\begin{equation}\label{eq:ssim}
    \textrm{SSIM}(F,G) \coloneqq \frac{2\mu_F\mu_G+c_1}{\mu^2_F+\mu^2_G+c_1}\cdot\frac{2\sigma_{F,G}+c_2}{\sigma^2_F+\sigma^2_G+c_2},
\end{equation}
where $\mu_F,\mu_G$ are the sample mean of $F$ and $G$, whilst $\sigma^2_F,\sigma^2_G$ and $\sigma_{F,G}$ are the sample variances and covariance. The 
constants $c_1, c_2>0$ are stabilizing factors that avoid division by zero, and can be tuned by the user.
The first term in the product \eqref{eq:ssim} is regarded as a luminance similarity between $F$ and $G$, and in fact (see e.g. \cite{Venkataramanan2021}) a more general version of the local SSIM can be defined, where the second factor of the product in \eqref{eq:ssim} is further split into two terms representing the contrast and structural similarity of the images. Since this approach is uncommon as stated in \cite{Venkataramanan2021}, we adopt the same assumption here and use the definition of Equation \eqref{eq:ssim}.

Differently with respect to the classic SSIM, the definition of the cSSIM is natively not restricted to matrices. We consider a bounded set 
$\Omega\subset\R^d$ and a probability measure $\nu$ on $\Omega$, and denote as $L_2^+(\Omega, \nu)$ the set of functions $f\in L_2(\Omega,\nu)$ with $f\geq 0$ $\nu$-almost everywhere ($\nu$-a.e. in the following). For $f, g\in L_2^+(\Omega,\nu)$ we define
\begin{align}\label{eq:mu_and_sigma}
\mu_f&\coloneqq \int_\Omega f d\nu = \norm{L_1}{f},\\
\sigma_{fg}&\coloneqq \int_\Omega (f-\mu_f) (g-\mu_g) d\nu = \inner{L_2}{f-\mu_f, g-\mu_g}\nonumber.
\end{align}
Then the cSSIM between $f$ and $g$ is defined in \cite{Marchetti2021a} as
\begin{align}\label{eq:cssim}
\cssim_{\nu}(f, g)\coloneqq \frac{2\mu_f\mu_g + c_1}{\mu_f^2 + \mu_g^2 + c_1}\cdot \frac{2\sigma_{fg} + c_2}{\sigma_{ff} + \sigma_{gg} + c_2},
\end{align}
with the constants $c_1, c_2>0$ playing the same role as in the discrete case. 

In the following, unless otherwise stated, we assume w.l.o.g. that $\nu$ is the normalized Lebesgue measure of $\Omega$, and we simply write $\cssim$, $L_2(\Omega)$, $L_2^+(\Omega)$.
Furthermore, observe that the cSSIM is defined in general dimension $d$, and it is thus not restricted to images. A similar extension would be possible in the discrete case by considering $d$-channel matrices, i.e., $F, G\in \R_{\geq 0} ^{p\times 
q\times d}$, and extending \eqref{eq:ssim} in the obvious way.

\subsection{A local definition}

In actual imaging problems, however, the SSIM has found wide application especially because it can be formulated in a local way in order to capture the fine-grained structure of an image, as opposite to more traditional metrics. 

For a detailed treatment of these aspects we refer to the recent paper \cite{Venkataramanan2021}, and we recall here only the basic facts needed for our analysis.
Following the definitions of the cited paper, given a window size $k\in\N$ we consider a $k\times k$ positive filter (or weight matrix) $W\in\R^{k\times k}_{\geq 0}$ with $\sum_{\ell,m=1}^k W(\ell, m)= 1$. For each pixel $(i, j)$ of an image $F\in\R_{\geq 0}^{p\times q}$, we further denote as $W_{ij}\in\R^{p\times q}$ the window obtained by shifting $W$ to the pixel $(i, j)$, i.e., 
\begin{equation*}
    W_{ij}(\ell, m):= 
    \begin{cases}
    W(\ell-i+1, m-j+1) &\text{ if } i\leq \ell\leq i+k-1, j\leq m \leq j+k -1, \\
    0 &\text{ otherwise}. 
    \end{cases}
\end{equation*} 
Observe that the convention here is to anchor the window on the corner pixel $(i,j)$ and to use a square support, but the extension to other centering and to non-square filters is straightforward, and we omit it here for simplicity.

Using this window one may define the local and weighted sample mean and variance of images $F,G\in\R_{\geq 0}^{p\times q}$. Namely, for $1\leq i\leq p$, $1\leq j\leq q$ we set
\begin{align}\label{eq:discrete_mu_sigma_weighted}
    \mu_F(i, j)&:= \sum_{\ell=1}^p\sum_{m=1}^q W_{ij}(\ell, m) F(\ell, m),\\
    \sigma_{FG}(i, j)&:= \sum_{\ell=1}^p\sum_{m=1}^q W_{ij}(\ell, m) \left(F(\ell, m)  - \mu_F(i, j)\right)\left(G(\ell, m)  - \mu_G(i, j)\right),\nonumber
\end{align}
and $\sigma_{F}^2(i, j):=\sigma_{FF}(i,j)$. Observe that in practice only the terms corresponding to the non-zero values of $W_{ij}$ are evaluated, so each sum is in fact over at most $k^2$ terms.

Given these local quantities, a local SSIM between $F, G\in \R_{\geq 0}^{p\times q}$ can be simply defined following \eqref{eq:ssim} as
\begin{equation}\label{eq:local_ssim}
Q(i, j) \coloneqq 
\frac{2\mu_F(i, j)\mu_G(i, j)+c_1}{\mu^2_F(i, j)+\mu^2_G(i, j)+c_1}
\cdot\frac{2\sigma_{FG}(i, j)+c_2}{\sigma^2_F(i, j)+\sigma^2_G(i, j)+c_2}, \;\; 1\leq i\leq p, 1\leq j\leq q.
\end{equation}
Finally, the weighted (or mean) SSIM is given by
\begin{equation}\label{eq:weighted_ssim}
\textrm{W-SSIM}(F,G):= \frac1{p q} \sum_{i=1}^p\sum_{j=1}^q Q(i, j).
\end{equation}
We remark that this is often employed as the standard definition of the SSIM, but we prefer to keep the name $\wssim$ in this paper to distinguish it from the globally defined version of Equation \eqref{eq:ssim}.

To analyze also this weighted index in the continuous case, we give the following definition that extends the one of \cite{Marchetti2021a} by using a convolution with a suitable weight function.
\begin{definition}\label{def:wcssim}
Let $\Omega\subset\R^d$ be bounded, $\nu$ be a probability measure on $\Omega$, and consider a continuous and positive weight function $w:\R^d\to \R$ such that $w(\cdot-x)\in L_2(\Omega,\nu)$ and $\int_\Omega w(y-x) d\nu(y) = 1$ for each $x\in \Omega$, and $w_{\max}\coloneqq\max_{x\in \Omega}{w(x)}<\infty$. 

For $f, g\in L_2^+(\Omega,\nu)$ and $x\in\Omega$, let
\begin{align}\label{eq:mu_and_sigma_weighted}
\mu_f(x)&\coloneqq \int_\Omega f(y) w(y- x) d\nu(y),\\
\nonumber\sigma_{fg}(x)&\coloneqq \int_\Omega (f(y)-\mu_f(x)) (g(y)-\mu_g(x)) w(y - x) d\nu(y).
\end{align}
Then the local cSSIM between $f$ and $g$ is defined for $c_1, c_2>0$ as
\begin{equation}\label{eq:local_cssim}
    q(x) \coloneqq \frac{2\mu_f(x)\mu_g(x) + c_1}{\mu_f(x)^2 + \mu_g(x)^2 + c_1}\cdot \frac{2\sigma_{fg}(x) + c_2}{\sigma_{ff}(x) + \sigma_{gg}(x) + c_2}, \;\; x\in\Omega,
\end{equation}
and the weighted cSSIM between $f$ and $g$ is defined as
\begin{align}\label{eq:weighted_cssim}
\wcssim_{\nu}(f, g)\coloneqq \int_\Omega q(x) d\nu(x).
\end{align}
\end{definition}

The interest is in the case when $w$ is compactly supported e.g. in a ball $B(0, r)$ for some (small) value $r>0$, so that $w(\cdot-x)$ is a local filter supported in $B(x, r)$. Moreover, also in this case, as in the discrete setting, it would be possible to consider a general weight function $w:\Omega\times \Omega\to\R$ and replace $w(y-x)$ with $w(x, y)$ in Equation \eqref{eq:mu_and_sigma_weighted}. This generalization does not change in any significant way the results of this paper, and so we prefer to stick to the simpler formulation given in Definition \ref{def:wcssim}.

\subsection{Relations between the different indices}\label{sec:relation}

The $\wcssim$ is clearly a generalization of the $\cssim$, since the particular choice $w\equiv 1$  gives $\mu_f(x) = \mu_f$ and $\sigma_{fg}(x) = \sigma_{fg}$ for each $x\in\Omega$ (Equation \eqref{eq:mu_and_sigma_weighted}). It follows that in this case $q(x)=\cssim(f, g)$ for all $x\in\Omega$ (Equation \eqref{eq:local_cssim}), and since $\nu$ is a probability measure we also have $\wcssim(f, g) = \cssim(f, g)$ from Equation \eqref{eq:weighted_cssim}.

On the other hand, observe that under the assumptions on $w$ (see Definition \ref{def:wcssim}), for each $x\in\Omega$ the measure $\nu_{w, x}$ defined for a measurable $f$ as 
\begin{equation*}
\int_\Omega f(y) d\nu_{w,x}(y)   :=\int_\Omega f(y) w(y -x) d\nu(y),
\end{equation*}
is a probability measure on $\Omega$. Moreover, observe that $f\in L_2^+(\Omega, \nu)$ implies that $f\in L_2^+(\Omega, \nu_{w,x})$ since
\begin{equation*}
\norm{L_2(\Omega, \nu_{w,x})}{f}^2=
    \int_\Omega |f(x)|^2 w(y-x) d\nu(x) \leq \max\limits_{x\in \Omega}{w(x)}\int_\Omega |f(x)|^2 d\nu(x)
=w_{\max}\norm{L_2(\Omega, \nu_{w,x})}{f}^2.
\end{equation*}
We will thus just assume that $f, g\in L_2^+(\Omega)\coloneqq L_2^+(\Omega,\nu)$ in the following. 
With this observation, it is immediate to see that $\mu_f(x)$ as defined in \eqref{eq:mu_and_sigma_weighted} is indeed $\mu_f$ of Equation \eqref{eq:mu_and_sigma} when computed with respect to the measure $\nu_{w, x}$, and similarly for $\sigma_{fg}(x)$. It follows that for all $x\in \Omega$ and $f, g\in L_2^+(\Omega)$ it holds $q(x) = \cssim_{\nu_{w,x}}(f, g)$, and thus
\begin{equation}\label{eq:wcssim_as_mean_cssim}
\wcssim_\nu(f, g) = \int_\Omega \cssim_{\nu_{w,x}}(f, g) d\nu(x),
\end{equation}
i.e., the $\wssim$ is an average over local $\cssim$s.

Moreover, it is actually the case that the continuous indices are indeed generalizations of the discrete ones in a very specific sense that we are formulating in the next proposition. To analyze this relation we consider for simplicity the two dimensional case, even if the extension to higher dimensions is straightforward.
The following proposition proves that the classical and SSIM can be obtained from the cSSIM when $\nu$ is a discrete counting measure. On the other hand, the discrete indices are a discretization of the continuous ones: if the continuous functions $f, g$ are discretized to images of increasing resolution, then the SSIM of the discretizations converge to the cSSIM of the original functions.
The general case of the weighted and local indices easily follows thanks to Equation \eqref{eq:wcssim_as_mean_cssim}.
\begin{proposition}\label{prop:ssimtocssim}
Let $\Omega\coloneqq[a,b]\times[c,d]\subset \mathbb{R}^2$, $a<b,\;c<d$, and let $f,g\in L_2^+(\Omega)$ be bounded and continuous on $\Omega$. Let 
$m,n\in\mathbb{N}$, $m=cn$ with $c\in\mathbb{Q}$, and let $p_0,\dots,p_m$, $q_0,\dots,q_n$ be such that $p_i\coloneqq c+i(d-c)/m$, $q_i\coloneqq a+i(b-a)/n$ and $p_{i+1}-p_i 
= q_{i+1}-q_i$ for $i=0,\dots,\min\{m,n\}$.
We define the sequences of matrices $\{F_{n}\}_{n\in\mathbb{N}}$, $\{G_{n}\}_{n\in\mathbb{N}}$ as
\begin{equation*}
    F_{n}(i,j)\coloneqq f(p_i,q_j),\; G_{n}(i,j)\coloneqq g(p_i,q_j),
\end{equation*}
$1\le i\le m-1$, $1\le j\le n-1$, so that $F_n$ and $G_n$ are $(m-1)\times(n-1)$ matrices. 
Moreover, let $\nu_n$ be the normalized counting measure supported on the points $\{(p_i, q_j)\}_{1\leq i\leq m-1, 1\leq j\leq n-1}$.

Then, we have
\begin{enumerate}
    \item $\cssim_{\nu_n}(f, g) = \mathrm{SSIM}(F_n, G_n)$.
    \item $\lim_{n\to\infty}{\mathrm{SSIM}(F_n,G_n)=\mathrm{cSSIM}(f,g)}$.
\end{enumerate}
\end{proposition}
\begin{proof}
The first point follows directly from the definition \eqref{eq:cssim} of the cSSIM.
To prove the second point we can associate to $F_n, G_n$ the piece-wise constant functions on $\Omega$
\begin{equation*}
    f_{n}\coloneqq\sum_{i,j=0}^{m,n}{f(p_i,q_j)\chi_{[p_i,p_{i+1}[\times [q_j,q_{j+1}[}},\quad\quad
    g_{n}\coloneqq\sum_{i,j=0}^{m,n}{g(p_i,q_j)\chi_{[p_i,p_{i+1}[\times [q_j,q_{j+1}[}},
\end{equation*}
where $\chi$ is the indicator function. Then, in this view, the sample means $\mu_{F_n}$ and $\mu_{G_n}$ of $F_n$ and $G_n$ are equivalent to the normalized Riemann integrals of $f_n$ and $g_n$. Since $f,g$ are bounded and continuous almost everywhere, such integrals converge to the normalized Lebesgue means $\mu_f$ and $\mu_g$ as $n$ tends to infinity.
\end{proof}

As a result of Proposition \ref{prop:ssimtocssim}, the cSSIM (or $\wcssim$) can be interpreted as the SSIM (or $\wssim$) computed on two infinite-resolution images. 
Alternatively, $\mathrm{SSIM}$ is a good approximation of the $\mathrm{cSSIM}$ as the resolution gets large enough. In this sense, the results of the following sections, and in particular the analysis of convergence of the $\cssim$, may be interpreted in terms of expected rate of approximability of an image when super-resolution techniques are applied.

\section{Bounding the cSSIM via the $L_2$ norm}\label{sec:cssim_l2}
We are now interested in linking the $\cssim$ to more classical error metrics, and especially the $L_2$ error. 

It turns out that the analysis greatly simplifies if one first analyzes the unweighted case, so we start with it in Section \ref{sec:bound_unweight}. The general result follows by manipulation of this basic case, and is derived in Section \ref{sec:bound_weight}.

\subsection{The global and unweighted case}\label{sec:bound_unweight}
As a first step, we consider the two multiplicative terms that define the cSSIM in \eqref{eq:cssim} and express them in terms of $L_1$ and $L_2$ norms. The 
following lemma is a simple adaptation of the argument in \cite[Section A]{Brunet12a}, i.e., we just rephrase the same computations using the $L_1$ and 
$L_2$ norm instead of the discrete $2$-norm used in the cited paper. 

\begin{lemma}\label{lemma:nmse}
For $f, g\in L_2^+(\Omega)$ and $c_1, c_2>0$ define
\begin{align*}
M(f, g)\coloneqq \frac{2\mu_f\mu_g + c_1}{\mu_f^2 + \mu_g^2 + c_1},
\qquad
S(f, g)\coloneqq\frac{2\sigma_{fg} + c_2}{\sigma_{ff} + \sigma_{gg} + c_2}.
\end{align*}
Then
\begin{align}
1 -\phantom{S}M(f, g) &\leq \frac{\norm{L_1}{f-g}^2}{\mu_f^2 + \mu_g^2 + c_1}\label{eq:one_minus_m}\\
1 -\phantom{M}S(f, g) &= \frac{\norm{L_2}{(f-\mu_f)-(g-\mu_g)}^2}{\sigma_{ff} + \sigma_{gg} + c_2}\label{eq:one_minus_s}.
\end{align}
\end{lemma}
\begin{proof}
By direct computation we have
\begin{align*}
1 - M(f, g)
&=1-\frac{2\mu_f\mu_g + c_1}{\mu_f^2 + \mu_g^2 + c_1}
=\frac{\mu_f^2 + \mu_g^2 + c_1 - 2\mu_f\mu_g - c_1}{\mu_f^2 + \mu_g^2 + c_1}
=\frac{(\mu_f - \mu_g)^2}{\mu_f^2 + \mu_g^2 + c_1}.
\end{align*}
Moreover, the numerator in last term can be bounded via
\begin{equation*}
\mu_f - \mu_g = \int_\Omega (f - g) dx\leq  \int_\Omega |f - g| dx = \norm{L_1}{f-g}, 
\end{equation*}
and this proves \eqref{eq:one_minus_m}. 
Similarly, for $S$ we obtain
\begin{align*}
1 - S(f, g)
&=1-\frac{2\sigma_{fg} + c_2}{\sigma_{ff} + \sigma_{gg} + c_2}
=\frac{\sigma_{ff} + \sigma_{gg} + c_2 - 2\sigma_{fg} - c_2}{\sigma_{ff} + \sigma_{gg} + c_2}
=\frac{\sigma_{ff} + \sigma_{gg} - 2\sigma_{fg}}{\sigma_{ff} + \sigma_{gg} + c_2},
\end{align*}
and in this case we have 
\begin{align*}
\sigma_{ff} + \sigma_{gg} - 2\sigma_{fg}
&=\norm{L_2}{f-\mu_f}^2 + \norm{L_2}{g-\mu_g}^2- 2\inner{L_2}{f-\mu_f, g-\mu_g}\\
&=\norm{L_2}{(f-\mu_f)-(g-\mu_g)}^2,
\end{align*}
which proves \eqref{eq:one_minus_s}.
\end{proof}
Next, we prove a simple bound on $M$ and $S$.
\begin{lemma}\label{lemma:s_smaller_than_one}
Let $f,g\in L_2^+(\Omega)$. Then
$0\leq M(f, g) \leq 1$, $|S(f, g)| \leq 1
$.
\end{lemma}
\begin{proof}
By the definition of $M$, the condition on is equivalent to show that $0\leq 2\mu_f\mu_g + c_1\leq \mu_f^2+\mu_g^2+c_1$, which is trivially true since $\mu_f, \mu_g, c_1\geq 0$ (for the lower bound), and $(\mu_f-\mu_g)^2\geq 0$ (for the upper bound).

The result for $S$ is equivalent by definition to prove that $\left|2\sigma_{fg} + c_2\right| \leq \sigma_{ff} + \sigma_{gg} + c_2$, and since 
$\left|2\sigma_{fg} + c_2\right|\leq 2\left|\sigma_{fg}\right| + c_2$ it is sufficient to prove that 
$2\left|\sigma_{fg}\right| + c_2\leq \sigma_{ff} + \sigma_{gg} + c_2$, or equivalently that $0 \leq \sigma_{ff} + \sigma_{gg}-2\left|\sigma_{fg}\right|$.
From the definition \eqref{eq:mu_and_sigma} of $\sigma_{fg}$ and the Cauchy-Schwartz inequality we have 
\begin{equation*}
|\sigma_{fg}| =  \left|\inner{L_2}{f-\mu_f, g-\mu_g}\right|\leq \norm{L_2}{f-\mu_f} \norm{L_2}{g-\mu_g},
\end{equation*}
and thus, again by definition of $\sigma_f$, $\sigma_g$, we obtain
\begin{align*}
\sigma_{ff} + \sigma_{gg} - 2\left|\sigma_{fg}\right|
&\geq \norm{L_2}{f-\mu_f}^2 + \norm{L_2}{g-\mu_g}^2 - 2 \norm{L_2}{f-\mu_f} \norm{L_2}{g-\mu_g}\\
&=\left(\norm{L_2}{(f-\mu_f)} - \norm{L_2}{(g-\mu_g)}\right)^2\geq 0,
\end{align*}
which concludes the proof.
\end{proof}

These lemmas allows us to derive the following key estimate. The result shows that the approximation error between two functions $f$ and $g$, as measured by 
the cSSIM, can be controlled by the squared $L_2$ distance of the two functions.

\begin{theorem}\label{thm:cssim_upper_bound}
Let $f,g\in L_2^+(\Omega)$ and $c_1, c_2>0$. Then it holds 
\begin{align*}
\left|1 - \cssim(f, g)\right|
&\leq c_{fg} \norm{L_2}{f-g}^2
\leq c_f \norm{L_2}{f-g}^2
\leq c \norm{L_2}{f-g}^2,
\end{align*}
where
\begin{align}\label{eq:constant_cf_cfg}
c_{fg}&\coloneqq\frac{4}{\sigma_{ff} + \sigma_{gg} + c_2} + \frac{1}{\mu_f^2 + \mu_g^2 + c_1},&
c_f&\coloneqq\frac{4}{\sigma_{ff} + c_2} + \frac{1}{\mu_f^2 + c_1},&
c&\coloneqq\frac{4 c_1 + c_2}{c_1c_2}.\;\; 
\end{align}
\end{theorem}
\begin{proof}
Since $\cssim(f, g) = M(f, g) \cdot S(f, g)$ by \eqref{eq:cssim}, and since $\left|S(f, g)\right|\leq 1$ by Lemma \ref{lemma:s_smaller_than_one}, it follows that 
\begin{align}\label{eq:plus_minus_s}
|1 - \cssim(f, g)|
&= \left|1 -  S(f, g) +  S(f, g)-  M(f, g) S(f, g)\right|\\
&= \left|1 - S(f, g)\right| + \left|S(f, g)\right| \left|1-  M(f, g)\right|\nonumber\\
&\leq \left|1 - S(f, g)\right| + \left|1-  M(f, g)\right|\nonumber,
\end{align}
where the second equality holds because all terms are already positive. Lemma \ref{lemma:nmse} thus gives
\begin{align}\label{eq:ssim_intermediate}
|1 - \cssim(f, g)|
&\leq \frac{\norm{L_2}{(f-\mu_f)-(g-\mu_g)}^2}{\sigma_{ff} + \sigma_{gg} + c_2} + \frac{\norm{L_1}{f-g}^2}{\mu_f^2 + \mu_g^2 + c_1}.
\end{align}
Now, since the measure is normalized, it holds $\norm{L_2}{c} = |c|$ for any constant function $c$, and by the H\"older inequality we have $\norm{L_1}{f}\leq 
\norm{L_2}{f}$. Moreover, $\left|\mu_f-\mu_g\right| = \norm{L_1}{f-g}$ since both functions are non-negative. It follows that
\begin{align*}
\norm{L_2}{(f-\mu_f)-(g-\mu_g)}
&\leq \norm{L_2}{f-g} + \norm{L_2}{\mu_f-\mu_g}
= \norm{L_2}{f-g} + \left|\mu_f-\mu_g\right|\\
&\leq \norm{L_2}{f-g} + \norm{L_1}{f-g}
\leq 2\norm{L_2}{f-g}.
\end{align*}
Combining this bound with \eqref{eq:ssim_intermediate}, and bounding again the $L_1$ norm with the $L_2$ norm in the second term,  gives
\begin{equation*}
\left|1 - \cssim(f, g)\right|
\leq \frac{4\norm{L_2}{f-g}^2}{\sigma_{ff} + \sigma_{gg} + c_2} + \frac{\norm{L_2}{f-g}^2}{\mu_f^2 + \mu_g^2 + c_1}
=c_{fg} \norm{L_2}{f-g}^2,
\end{equation*}
where $c_{fg}\coloneqq4/(\sigma_{ff} + \sigma_{gg} + c_2) + 1/(\mu_f^2 + \mu_g^2 + c_1)$. Finally 
\begin{equation*}
c_{fg}\leq\frac{4}{\sigma_{ff} + c_2} + \frac{1}{\mu_f^2 + c_1}=:c_f,
\end{equation*}
and since $\mu_f, \sigma_{ff}\geq 0$ we have $c_f\leq {4}/{c_2} + {1}/{c_1}:=c$, and the proof is complete.
\end{proof}

\subsection{The local and weighted case}\label{sec:bound_weight}
Using the relation between the $\wcssim$ and the $\cssim$ outlined in Section \ref{sec:relation}, and especially Equation \eqref{eq:wcssim_as_mean_cssim}, we 
are in the position to prove the following corollary of Theorem \ref{thm:cssim_upper_bound}.

\begin{corollary}\label{cor:wcssim_upper_bound}
Under the assumptions of Definition \ref{def:wcssim}, let $f,g\in L_2^+(\Omega)$. Then 
\begin{equation*}
\left|1 - \wcssim(f, g)\right|
\leq C_{fg} \norm{L_2(\Omega)}{f-g}^2
\leq C_f \norm{L_2(\Omega)}{f-g}^2
\leq c \norm{L_2(\Omega)}{f-g}^2,
\end{equation*}
where $c$ is defined as in Theorem \ref{thm:cssim_upper_bound} and
\begin{align*}
C_{fg}&\coloneqq\max\limits_{x\in\Omega}\frac{4}{\sigma_{ff}(x) + \sigma_{gg}(x) + c_2} + \frac{1}{\mu_f(x)^2 + \mu_g(x)^2 + c_1},\\ 
C_f&\coloneqq\max\limits_{x\in\Omega}\frac{4}{\sigma_{ff}(x) + c_2} + \frac{1}{\mu_f^2(x) + c_1}.\nonumber 
\end{align*}
\end{corollary}
\begin{proof}
Since for all $x\in\Omega$ the measure $\nu_{w, x}$ is a probability on $\Omega$, and since $f, g\in L_2^+(\Omega, \nu_{w, x})$, we can apply Theorem \ref{thm:cssim_upper_bound} to $q(x) = \cssim_{\nu_{w,x}}(f, g)$ to obtain
\begin{equation*}
\left|1 - q(x)\right|
\leq c_{fg}(x) \norm{L_2(\Omega, \nu_{w, x})}{f-g}^2\;\;\fa g\in L_2^+(\Omega),\;\fa x\in \Omega,  \end{equation*}
where we denote as $c_{fg}(x)$ the constant 
\begin{equation*}
c_{fg}(x)\coloneqq\frac{4}{\sigma_{ff}(x) + \sigma_{gg}(x) + c_2} + \frac{1}{\mu_f(x)^2 + \mu_g(x)^2 + c_1}.
\end{equation*}
Using this inequality together with the definition of the $\wcssim$, and using the fact that $\nu$ integrates to one, we obtain
\begin{align}\label{eq:cor_tmp}
\left|1 - \wcssim(f, g)\right| 
&= \left|1 - \int_\Omega q(x) d \nu(x)\right|
= \left|\int_\Omega (1 - q(x)) d \nu(x)\right|\nonumber\\
&\leq \int_\Omega \left|1 - q(x)\right| d \nu(x)
\leq \int_\Omega c_{fg}(x) \norm{L_2(\Omega, \nu_{w, x})}{f-g}^2 d \nu(x)\nonumber\\
&\leq \left(\max\limits_{x\in\Omega}c_{fg}(x)\right) \int_\Omega \norm{L_2(\Omega, \nu_{w, x})}{f-g}^2 d \nu(x).
\end{align}
Setting $C_{f, g}:=\max\limits_{x\in\Omega}c_{fg}(x)$, we have from the same estimate as in Theorem \ref{thm:cssim_upper_bound} that 
\begin{align*}
C_{f,g}
&\leq C_f
\coloneqq \max\limits_{x\in\Omega}\left(\frac{4}{\sigma_{ff}(x) + c_2} + \frac{1}{\mu_f(x)^2 + c_1}\right)\leq c.
\end{align*}
To estimate the second term we have instead
\begin{align*}
\int_\Omega \norm{L_2(\Omega, \nu_{w, x})}{f-g}^2 d \nu(x)
&= \int_\Omega \int_\Omega (f(y) - g(y))^2 w(y - x) d\nu(y) d\nu(x)\\
&= \int_\Omega (f(y) - g(y))^2 \int_\Omega w(y - x) d\nu(x) d\nu(y)\\
&= \int_\Omega (f(y) - g(y))^2 d\nu(y)
=\norm{L_2(\Omega, \nu)}{f-g}^2,
\end{align*}
since $\int_\Omega w(y - x) d\nu(x)=1$ for all $x$. 
Inserting these two terms in \eqref{eq:cor_tmp} concludes the proof.
\end{proof}

\subsection{Conditions for equivalence}\label{sec:equivalence}
At this point one may ask if an inverse inequality holds too, proving that the two error measures are equivalent up to appropriate scaling factors. In this  section we prove that this is indeed the case, but only under additional assumptions.
We have the following.
\begin{theorem}\label{thm:equivalence}
Let $f, g\in L_2^+(\Omega)$, and let $R>0$ be such that 
$\norm{L_2(\Omega, \nu_{w, x})}{f}, \norm{L_2(\Omega, \nu_{w, x})}{g}\leq R$ for all $x\in\Omega$ (e.g., $\norm{L_2}{f},\norm{L_2}{g}\leq R$).
Then it holds
\begin{equation}\label{eq:lower_bound}
\frac{1}{4 R^2 + c_2} \left(\norm{L_2}{f-g}^2 - \int_\Omega\left(\mu_f(x) - \mu_g(x)\right)^2 d\nu(x)\right) \leq \left|1-\wcssim(f,g)\right|.
\end{equation}
In particular, if there exists $c'\in\left(0, 1/(4R^2 + c_2)\right]$ such that
\begin{equation}\label{eq:condition_lower_bound}
\int_\Omega\left(\mu_f(x) - \mu_g(x)\right)^2 d\nu(x)
\leq\left(1- c' (4 R^2 + c_2)\right) \norm{L_2}{f-g}^2,
\end{equation}
then
\begin{equation}\label{eq:lower_bound_equivalence}
c' \norm{L_2}{f-g}^2 \leq \left|1-\wcssim(f,g)\right|,
\end{equation}
i.e., the two measures are equivalent. 
Moreover, if $w_{\max}<1$ then \eqref{eq:condition_lower_bound} holds with  
\begin{equation}\label{eq:condition_lower_bound_w}
c'\coloneqq\frac{1-w_{\max}^2}{4R^2 + c_2}>0.
\end{equation}
\end{theorem}
\begin{proof}
Applying the identity part of inequality \eqref{eq:plus_minus_s} to $\cssim_{\nu_{w,x}}$, we have 
\begin{equation*}
|1 - q(x)| = |1 - S(x)| + |S(x)| |1 - M(x)|
\geq |1 - S(x)|\;\;\fa x\in\Omega,
\end{equation*}
where $S(x)$ and $M(x)$ are the local and weighted versions of $S(f, g)$ and $M(f, g)$.

Furthermore, we can rewrite Equation \eqref{eq:one_minus_s} using the $L_2(\Omega, \nu_{w,x})$-inner product to get
\begin{equation*}
1 - S(x) 
= \frac{\norm{L_2(\Omega, \nu_{w,x})}{(f-\mu_f(x))-(g-\mu_g(x))}^2}{\sigma_{ff}(x) + \sigma_{gg}(x) + c_2},    
\end{equation*}
and thus, setting $h:=f-g$, for all $x\in\Omega$ it holds
\begin{align}\label{eq:cor_intermediate}
|1 - \wcssim(f, g)| 
&\geq \int_\Omega |1 - S(x)| d\nu(x)
\geq \int_\Omega\frac{\norm{L_2(\Omega, \nu_{w,x})}{h-\mu_h(x)}^2}{\sigma_{ff}(x) + \sigma_{gg}(x) + c_2} d\nu(x)\nonumber\\
&\geq 
\left(\min\limits_{x\in\Omega}\frac1{\sigma_{ff}(x) + \sigma_{gg}(x) + c_2} \right)\int_\Omega \norm{L_2(\Omega, \nu_{w,x})}{h-\mu_h(x)}^2 d\nu(x),
\end{align}
since all terms are non negative. Now for the last term we have
\begin{align*}
\int_\Omega \norm{L_2(\Omega, \nu_{w,x})}{h-\mu_h(x)}^2 d\nu(x)
&=\int_\Omega \left(\norm{L_2(\Omega, \nu_{w,x})}{h}^2 - \mu_h(x)^2\right) d\nu(x)\\
&=\norm{L_2(\Omega, \nu}{h}^2 - \int_\Omega\mu_h(x)^2 d\nu(x).
\end{align*}
Moreover, \eqref{eq:mu_and_sigma} gives 
\begin{align*}
\sigma_{ff}(x) 
&= \norm{L_2(\Omega,\nu_{w,x})}{f-\mu_f(x)}^2 
= \norm{L_2(\Omega,\nu_{w,x})}{f}^2 - \mu_f(x)^2 
\leq \norm{L_2(\Omega,\nu_{w,x})}{f}^2 + \mu_f(x)^2\\ 
&= \norm{L_2(\Omega,\nu_{w,x})}{f}^2 + \norm{L_1(\Omega,\nu_{w,x})}{f}^2
\leq 2\norm{L_2(\Omega,\nu_{w,x})}{f}^2
\leq 2 R^2,
\end{align*}
and similarly for $\sigma_{gg}(x)$. It follows that ${\sigma_{ff}(x) + \sigma_{gg}(x) + c_2}\leq 4 R^2 + c_2$ for all $x\in\Omega$, and inserting this upper bound in \eqref{eq:cor_intermediate} concludes the proof of \eqref{eq:lower_bound}.

Now, inserting the condition \eqref{eq:condition_lower_bound} into \eqref{eq:lower_bound} one obtains that 
\begin{align*}
\frac{1}{4 R^2 + c_2} &\left(\norm{L_2}{f-g}^2 - \int_\Omega\left(\mu_f(x) - \mu_g(x)\right)^2 d\nu(x)\right)\\
&\geq \frac{1}{4 R^2 + c_2} \left(1 - (1 - c'(4R^2 + c_2)\right)\norm{L_2}{f-g}^2
\geq c'\norm{L_2}{f-g}^2,
\end{align*}
which proves \eqref{eq:lower_bound_equivalence}. 

Finally, if $w_{\max}<1$ then clearly the constant $c'$ defined in \eqref{eq:condition_lower_bound_w} satisfies $0<c'<1/(4R^2 + c_2)$. It thus remains to prove that \eqref{eq:condition_lower_bound} is valid for this $c'$. This follows by direct computation, since $\nu$ is a probability measure and thus
\begin{align*}
\int_\Omega\left(\mu_f(x) - \mu_g(x)\right)^2 d\nu(x)
&=
\int_\Omega\left(\int_\Omega\left(f(y) - g(y)\right)w(y-x)d\nu(y)\right)^2 d\nu(x)\\
&\leq \int_\Omega\left(\int_\Omega\left|f(y) - g(y)\right|w(y-x)d\nu(y)\right)^2 d\nu(x)\\
&\leq \int_\Omega w_{\max}^2\norm{L_1(\Omega, \nu)}{f-g}^2 d\nu(x)\\
&= w_{\max}^2\norm{L_1(\Omega, \nu)}{f-g}^2
\leq w_{\max}^2\norm{L_2(\Omega, \nu)}{f-g}^2,
\end{align*}
and therefore \eqref{eq:condition_lower_bound} holds provided that $w_{\max}^2\leq (1-c'(4R^2+c_2)$, which holds with equality by the definition of $c'$.
\end{proof}

The theorem proves that the $\wcssim$ is equivalent to the $L_2$ norm under two additional conditions, namely the bound on $\norm{L_2}{f}$, $\norm{L_2}{g}$ and condition \eqref{eq:condition_lower_bound}.

These two conditions are quite different in nature. Indeed, since $\nu$ is a probability we have $\norm{L_2(\Omega)}{f} \leq \norm{L_\infty(\Omega)}{f}$, and 
thus for $f, g$ that represent images with bounded values it is immediate to find $R>0$ such that $\norm{L_2(\Omega,\nu_{w,x})}{f}\leq \norm{L_2(\Omega)}{f}\leq R$, and similarly for $g$. This is thus a not very restrictive requirement. On the other hand, condition \eqref{eq:condition_lower_bound} is not always satisfied, and we elaborate on its consequences in the following.

Nevertheless, we first point out that the theorem has in any case consequences on Theorem 
\ref{thm:cssim_upper_bound} and Corollary \ref{cor:wcssim_upper_bound}. Indeed, the condition \eqref{eq:condition_lower_bound} is easily met with the optimal constant $c':=1/(4 R^2 + c_2)$ if $f$ and $g$ are such that $\int_\Omega\left(\mu_f(x) - \mu_g(x)\right)^2 d\nu(x)=0$ (e.g., $\mu_f=\mu_g$ almost everywhere in $\Omega$).  
This in particular implies that the rates of Theorem 
\ref{thm:cssim_upper_bound} and Corollary \ref{cor:wcssim_upper_bound} (i.e., the exponent $2$ in the $L_2$ norm) are optimal, in the sense that they can not be improved if they have to hold for general $f, g\in L_2^+(\Omega)$. 

On the other hand, the condition \eqref{eq:condition_lower_bound} is not always verified, and indeed in general we can not expect an equivalence to hold. This means that there are cases when $\norm{L_2}{f-g}^2$ is converging to zero at a slower rate than $|1-\wssim(f, g)|$, or even not converging to zero at all. 

This fact, which is relevant in practice and not a novelty, is the fundamental reason why the SSIM is often preferred to the $L_2$ norm in imaging applications.
However, the lower bound of Theorem \ref{thm:equivalence} now gives a concrete hint on the reason why this situation may occur. Namely, since $\int_\Omega w(y-x)d\nu(y)=1$ by Definition \ref{def:wcssim}, to have a value $w_{\max}<1$ one needs to have a weight $w$ with wide enough support. 
In other words, to observe a practical benefit in using the SSIM instead of the $L_2$ error, a local enough weight should be chosen.

\begin{remark}
We point out that if a function $f\in L_2^+(\Omega)$ is approximated in the $L_2$ sense by a sequence $\{f_n\}\subset L_2^+(\Omega)$, since as observed before it holds 
\begin{equation*}
|\mu_f(x) -\mu_{f_n}(x)|
\leq \norm{L_1(\Omega,\nu_{w,x})}{f-f_n}
\leq \norm{L_2(\Omega,\nu_{w,x})}{f-f_n},
\end{equation*}
then we have $\mu_{f_n}(x)\to \mu_f(x)$. In this sense, we should expect to see an equivalence in the spirit of Theorem \ref{thm:equivalence} when the two 
error measures are used on $L_2$ converging approximations. We will verify this intuition in the numerical experiments.
Observe moreover that in the case of the $\cssim$ (i.e., $w\equiv 1$) the condition \eqref{eq:condition_lower_bound} is met if the global means satisfy $\mu_f=\mu_g$, which is a not very demanding assumption. 
\end{remark}

\section{Rates of convergence of the cSSIM for concrete image interpolation methods}\label{sec:int_methods}

In the following, we exploit the theoretical findings presented in Section \ref{sec:cssim_l2} and we provide some convergence results in terms of the cSSIM for commonly used image interpolation methods, i.e., we bound the theoretically expected cSSIM as the resolution gets larger. Completely analogous statements can be derived for the $\wcssim$.

First, for the case of gridded data we consider the bilinear interpolation method and the Hermite bicubic interpolation method. These are extensively used in image interpolation, and there exist state-of-the art implementations that do not require the knowledge of derivatives, and which can exploit several strategies to improve the computation of the interpolants. Nevertheless, we want to omit the error due to derivative approximation, and thus we provide explicit definitions of the algorithms in the following, which are equivalent to those standard ones, even if possibly less efficient.
Then, we consider the framework of kernel-based interpolation, which is particularly suitable for scattered data.

We consider for simplicity $\Omega=[a,b]\times[c,d]\subset\mathbb{R}^2$. We are interested in obtaining bounds on the convergence of these interpolation methods with respect to the cSSIM from known results on their $L_2$ convergence. The latter results are usually formulated in terms of the smoothness of the target function $f$, and to this end we will assume in the following that $f$ is an element of the Sobolev space $W_2^{\tau}(\Omega)$ of suitable fractional or integer order $\tau>0$ (see e.g. Chapter 3 in \cite{McLean2000}).
Moreover, we denote as $C^{k,l}(\Omega)$ the set of functions whose derivatives $\partial^{i+j} f/(\partial x^i \partial y^j)$ are continuous on $\Omega$ for 
$0\le i \le k$ and $0\le j \le l$.

The rate of approximation of the various methods are quantified in terms of the density and distribution of the interpolation points $X\subset \Omega$, that is quantified by means of the fill distance 
\begin{equation}
h_{X, \Omega} \coloneqq \sup\limits_{y\in\Omega} \min\limits_{x \in X}\|x-y\|,   
\end{equation}
which is a generalization of the grid size that is suitable also for scattered meshes.

\subsection{Bilinear interpolation}

Let $f\in C^{0,0}(\Omega),\;f\in W_2^2(\Omega)$, and let $E_{m,n}$ be a $m\times n$ equispaced grid of interpolation nodes in $\Omega$ with equal horizontal and vertical spacing $s>0$. Consequently, the fill distance is $h_{E_{m,n},\Omega}\coloneqq \sqrt{2}s/2$ and it represents in this case the pixel size. The unique bilinear interpolant $f_b$ of $f$ at $E_{m,n}$ is constructed upon $2\times 2$ local neighborhoods. More precisely, letting $\boldsymbol{z}_{11}=(x_1,y_1),\;\boldsymbol{z}_{21}=(x_2,y_1),\; \boldsymbol{z}_{12}=(x_1,y_2),\; \boldsymbol{z}_{22}=(x_2,y_2)\in E_{m,n}$ be the neighborhood nodes related to an evaluation point $\boldsymbol{\xi}=(\xi_1,\xi_2)\in\Omega$, that is $\boldsymbol{\xi}$ is contained in the rectangle defined by the vertices $\boldsymbol{z}_{11},\boldsymbol{z}_{21},\boldsymbol{z}_{12},\boldsymbol{z}_{22}$, we have
\begin{equation*}
    f_b(\xi_1,\xi_2)=c_0+c_1\xi_1+c_2\xi_2+c_3\xi_1\xi_2,
\end{equation*}
where the coefficients $c_0,\dots,c_3$ are determined by solving the linear system
\begin{equation}\label{eq:linsys_bil}
    \begin{bmatrix} 
1 & x_1 & y_1 & x_1y_1 \\
1 & x_1 & y_2 & x_1y_2 \\
1 & x_2 & y_1 & x_2y_1 \\
1 & x_2 & y_2 & x_2y_2 \\
\end{bmatrix}
\begin{bmatrix} 
c_0 \\
c_1 \\
c_2 \\
c_3 \\
\end{bmatrix}
=
\begin{bmatrix} 
f(\boldsymbol{z}_{11}) \\
f(\boldsymbol{z}_{21}) \\
f(\boldsymbol{z}_{12}) \\
f(\boldsymbol{z}_{22}) \\
\end{bmatrix}.
\end{equation}
The $L_2$-error between $f$ and $f_b$ on $\Omega$ can be bounded as (see e.g. \cite{Getreuer11})
\begin{equation*}
    \lVert f-f_b \lVert_{L_2}\le C h \lVert f \lVert_{W_2^{2}},
\end{equation*}
where $C>0$ is a constant independent of $h$. Therefore, from Theorem \ref{thm:cssim_upper_bound} we get
\begin{equation*}
   \left|1 - \cssim(f, f_b)\right| \le c_f C^2 h^2 \lVert f \lVert^2_{W_2^{2}}.
\end{equation*}

\subsection{Hermite bicubic interpolation}\label{sec:hermite}
Differently with respect to the bilinear interpolation case, the Hermite bicubic interpolant $f_c$ is built upon $4\times 4$ neighborhoods. In addition to the constraints imposed by the function values $f(\boldsymbol{z}_{11}),f(\boldsymbol{z}_{21}),f(\boldsymbol{z}_{12}),f(\boldsymbol{z}_{22})$, the function $f_c$ also interpolates $\partial f/ (\partial x)$, $\partial f/ (\partial y)$ and $\partial^2 f/ (\partial x\partial y)$ at $\boldsymbol{z}_{11},\boldsymbol{z}_{21},\boldsymbol{z}_{12},\boldsymbol{z}_{22}$. 
The extremal nodes of the neighborhood are indeed exploited in order to estimate such derivatives at the $2\times 2$ internal nodes. 

Therefore, the interpolant $f_c$ evaluated at $\boldsymbol{\xi}$ takes the form
\begin{equation*}
    f_c(\xi_1,\xi_2)=\sum_{0\le i,j\le 3}{c_{ij}\frac{\partial^{i+j} f}{\partial x^i \partial y^j}(\xi_1,\xi_2)},
\end{equation*}
where the coefficients $c_{ij}$ are determined by solving the linear system related to the interpolation task, similarly to \eqref{eq:linsys_bil}.

For $f\in C^{1,1}(\Omega)\cap W_2^4(\Omega)$, the paper \cite{Bialecki94} gives the bound
\begin{equation*}
    \lVert f-f_c \lVert_{L_2}\le C h^3 \lVert f \lVert_{W_2^{4}},
\end{equation*}
and thus by virtue of Theorem \ref{thm:cssim_upper_bound} we obtain
\begin{equation*}
   \left|1 - \cssim(f, f_c)\right| \le c_f C^2 h^6 \lVert f \lVert^2_{W_2^{4}}.
\end{equation*}
\subsection{Kernel-based interpolation}
Let $K:\Omega\times\Omega \to\R$ be a strictly positive definite kernel, i.e., for any set $X\subset \Omega$ of pairwise distinct points the kernel matrix $\mathsf{K}=(\mathsf{K}_{i,j})=K(\bs{x}_i,\bs{x}_j)$, $x_i, x_j\in X$ is positive definite. 
The kernel interpolant $f_k$ of $f\in C(\Omega)$ at $n$ pairwise distinct points $X\subset\Omega$ is defined as
\begin{equation*}
    f_k(\xi_1,\xi_2) = \sum_{i=1}^n{\alpha_i K(\bs{\xi},\bs{x}_i)},\quad \bs{x}_i\in X, \bs{\xi}\in \Omega,
\end{equation*}
with coefficients $\bs{\alpha}=(\alpha_1,\dots,\alpha_n)^{\intercal}\in\mathbb{R}^n$ that solve the linear system $\mathsf{K}\bs{\alpha}=\bs{f}$,
where $\mathsf{K}$ is the kernel matrix on $X$ and $\bs{f}=(f(\bs{x}_1),\dots,f(\bs{x}_n))^{\intercal}$.
It is known that if the kernel is additionally translational invariant, i.e., $K(\bs{x}, \bs{y})\coloneqq\Phi(\bs{x}-\bs{y})$ for some $\Phi:\R^d\to\R$, then under certain assumptions one can show that there exists $\tau>d/2$ such that for each $f\in W_2^\tau(\Omega)$ it holds 
\begin{equation*}
\norm{L_2}{f-f_k} \leq C h^{\tau} \norm{W_2^\tau}{f},
\end{equation*}
with a constant $C>0$ independent of $f$. We remark that the precise value of $\tau$ is given by the rate of polynomial decay of the Fourier transform of $\Phi$, and this value is connected to the smoothness of $K$. We do not give further explanations here, and we refer to \cite{Wendland2005} for a detailed discussion.
In this setting Theorem \ref{thm:cssim_upper_bound} gives that for each $f\in L_2^+(\Omega)\cap W_2^\tau(\Omega)$ it holds 
\begin{equation}\label{eq:sobolev_bound}
\left|1 - \cssim(f, f_k)\right| \leq c_f C^2  h^{2\tau} \norm{W_2^\tau}{f}^2.
\end{equation}
As notable examples, we report in particular the error bounds related to the two kernels considered in \cite{Marchetti2021a}. We remark that these new bounds are strict improvements over the results proven in that paper. We have the following:
\begin{itemize}
\item If $K$ is a $(d, k)$-Wendland kernel with $d=2$ and $k=1$, i.e., $\Phi(\bs{x}) = \varphi_1(\|\bs{x}\|)$ with $\varphi_1(r) = (1-r)_+^4 (4 r + 1)$, then $\tau = d/2 + k + 1/2 = 1 + 1 + 1/2 =5/2$ (see \cite{Wendland2005}), and thus \eqref{eq:sobolev_bound} gives
\begin{equation*}
\left|1 - \cssim(f, f_k)\right| \leq c_f C^2 h^{5} \norm{W_2^{5/2}}{f}^2.
\end{equation*}
\item If $K$ is a cubic Mat{\'e}rn kernel, i.e., $\Phi(\bs{x}) = \varphi_2(\|\bs{x}\|)$ with $\varphi_2(r) = e^{-r}(15 + 15r + 6r^2+r^3)$, then $\tau=(d+7)/2=9/2$ (see \cite{Wendland2005}), and thus \eqref{eq:sobolev_bound} gives
\begin{equation*}
\left|1 - \cssim(f, f_k)\right| \leq c_f C^2 h^{9} \norm{W_2^{9/2}}{f}^2.
\end{equation*}
\end{itemize}

Finally, we recall that it is common in practice to solve a regularized interpolation problem with kernel matrix $(\mathsf{K} + \lambda \mathsf{I})$ with $\lambda>0$ in order to improve the numerical conditioning of the system. Results are known to select $\lambda$ small enough such that the same convergence rates as in \eqref{eq:sobolev_bound} are obtained (see \cite{WendlandRieger2005}).

\section{Numerical tests}\label{sec:num_tests}
The experiments are carried out in \textsc{Python 3.6} and the code to replicate the examples is publicly available\footnote{\href{https://github.com/GabrieleSantin/cssim_convergence}{https://github.com/GabrieleSantin/cssim\_convergence}.}.
To verify the convergence estimates discussed in Section \ref{sec:int_methods}, we perform numerical tests taking the same examples studied in \cite{Marchetti2021a}, i.e. the functions $f_1,f_2:\Omega\longrightarrow\mathbb{R}$, $\Omega=[0,1]^2$, defined as
\begin{equation*}
    f_1(\bs{x})\coloneqq2(x_1x_2)^2-\sinc(x_1)\sinc(x_2)+1, \quad \quad f_2(\bs{x})\coloneqq e^{-(x_1+x_2)}-3x_1+x_2+5.
\end{equation*}
As interpolation sets we take equally spaced gridded data $X_i\subset \Omega$, $i=1,\dots,4$ with steps $s_i=2^{-i+2}/5$, i.e., $s_i = 0.4 , 0.2 , 0.1 , 0.05$. The reconstructions are then evaluated on a finer grid with step $10^{-2}$. Moreover, following Proposition \ref{prop:ssimtocssim}, to approximate the cSSIM we compute the SSIM on the evaluation grid, and then we express the results in terms of the dissimilarity $1-\mathrm{SSIM}$ and its weighted counterpart $1-\wssim$. The weighted index is computed with a $22\times 22$ constant weight matrix anchored at its center and normalized to unit sum.

We report in Figure \ref{fig:funct_interp} the results obtained by performing bilinear and bicubic interpolation, and kernel-based interpolation using $\varphi_1$ (the Wendland kernel) and $\varphi_2$ (the cubic Mat{\'e}rn kernel). 
Each figure shows the decay of the SSIM- and $\wssim$-dissimilarity indices, and compare them with the decay of the squared $L_2$ error for an increasing number of interpolation points.
As discussed before, in these experiments we use an own implementation of the interpolation methods, and in particular we use the exact values of the partial derivatives of the functions in the case of Hermite bicubic interpolation, in order to verify the investigated theoretical bounds. 
At a first glance, it is immediate to see that the dissimilarity indices decay at essentially the same rate of the $L_2$ error, both for the polynomial and kernel methods, thus confirming our theory.
We remark that in the case of the $\wssim$ for the bilinear interpolation of $f_1$ (left figure in Panel \ref{fig:bicunear_f1}) we have to extend the interpolation also to $s_i$ with $i=5, \dots, 9$ in order to observe this asymptotic rate.

\begin{figure}
\begin{center}
\begin{subfigure}{0.49\textwidth}
\includegraphics[width=\textwidth]{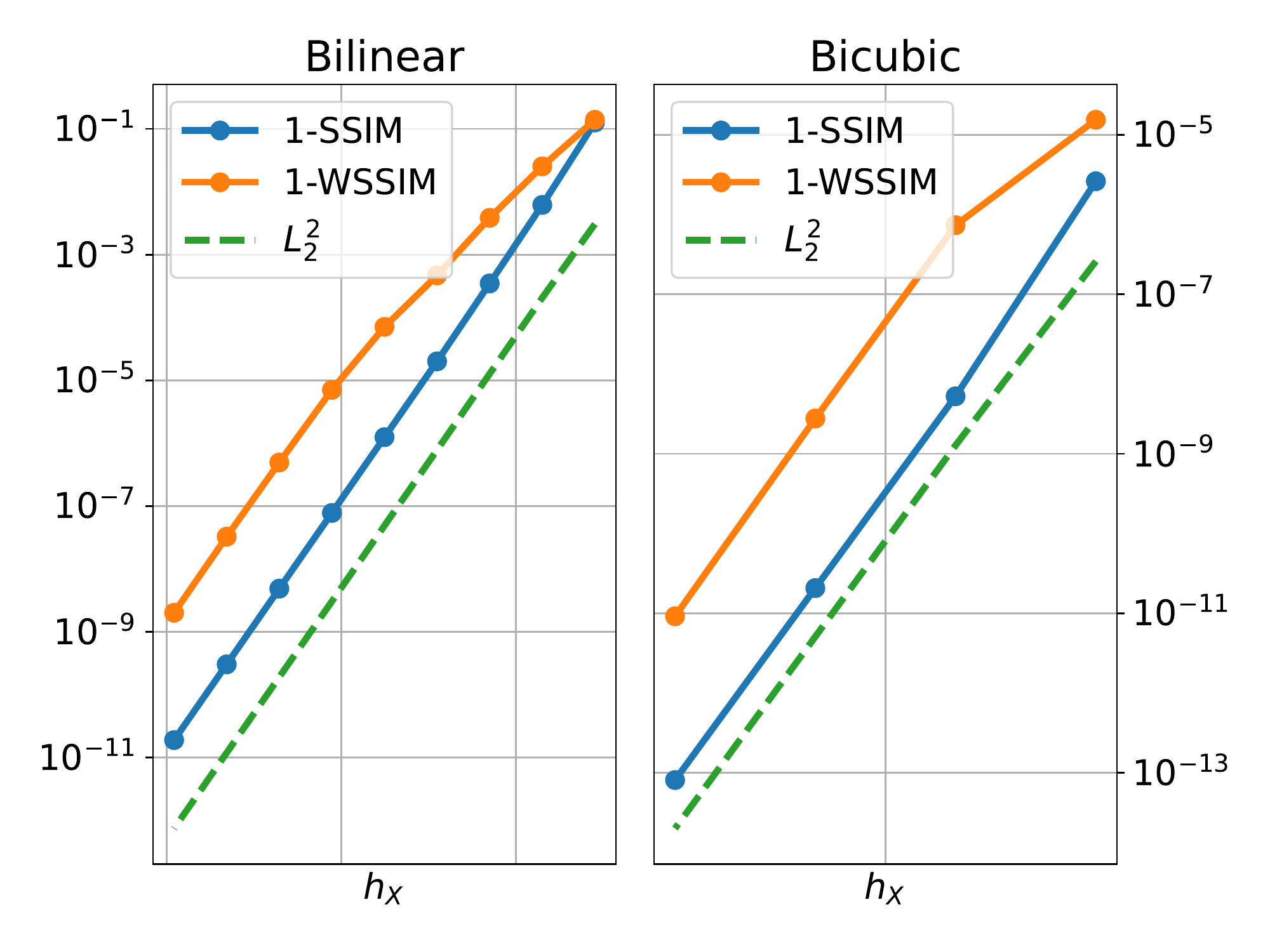}
\caption{}\label{fig:bicunear_f1}
\end{subfigure}
\begin{subfigure}{0.49\textwidth}
\includegraphics[width=\textwidth]{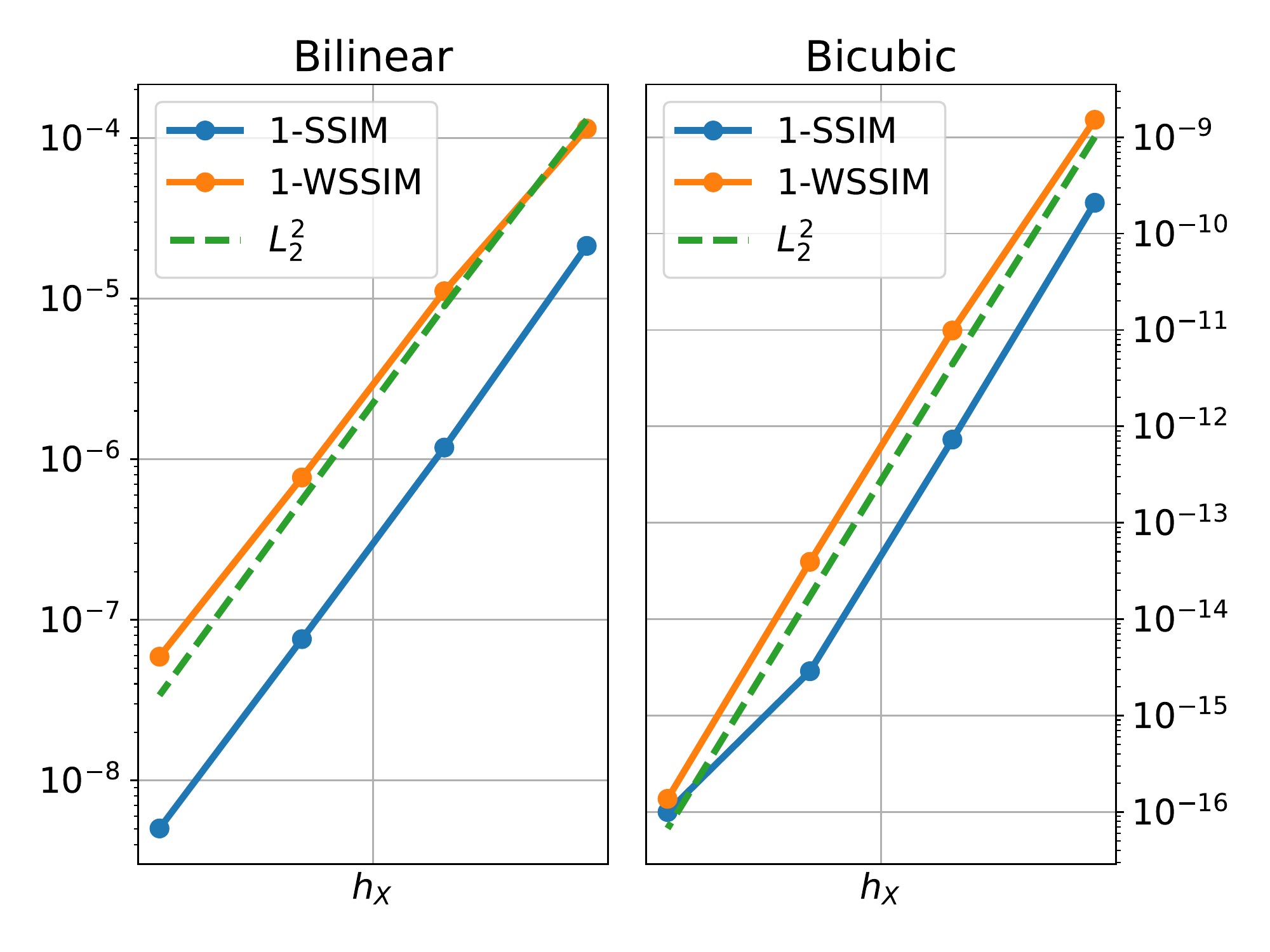}
\caption{}\label{fig:bicunear_f2}
\end{subfigure}
\begin{subfigure}{0.49\textwidth}
\includegraphics[width=\textwidth]{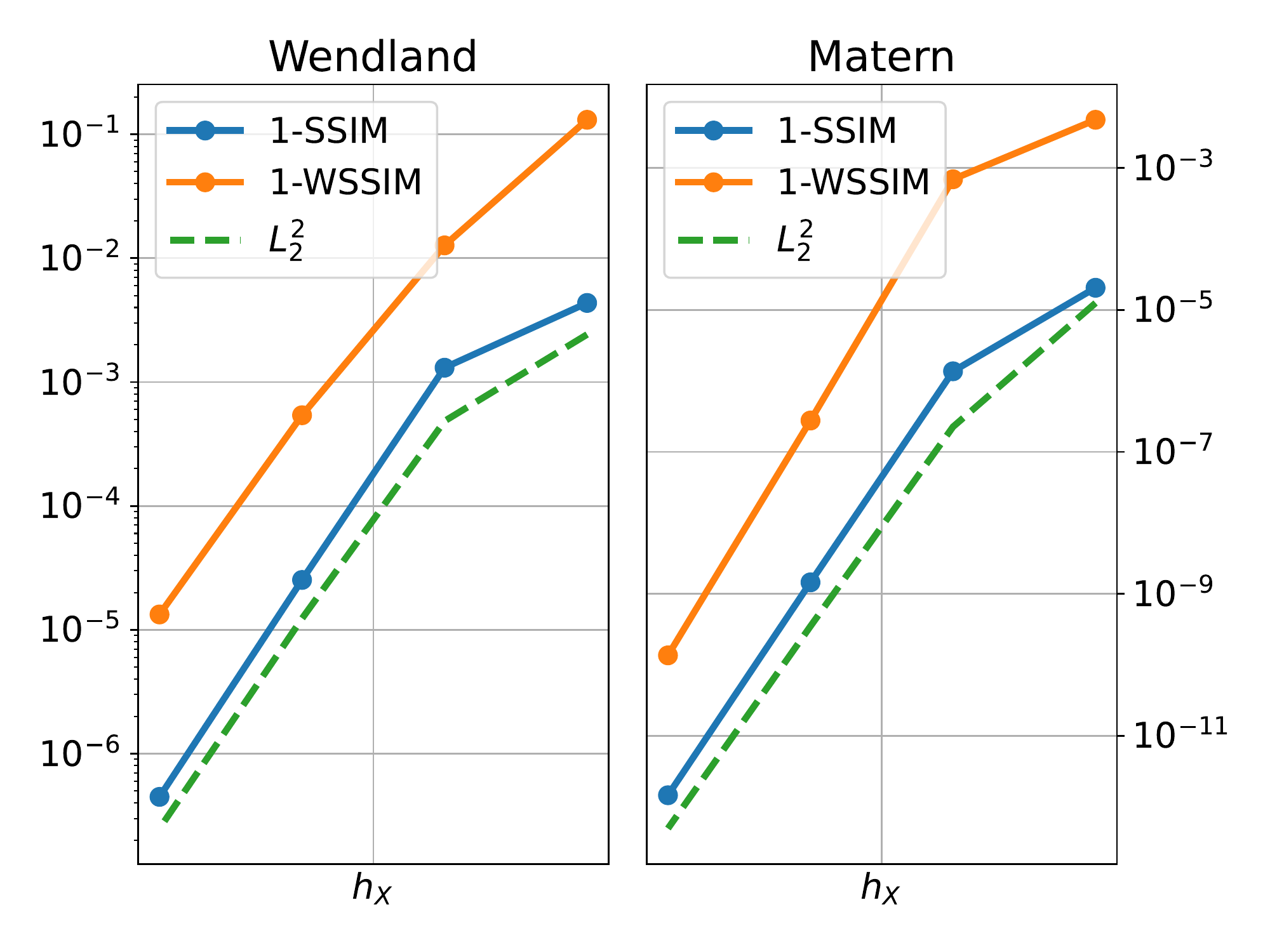}
\caption{}\label{fig:kernel_f1}
\end{subfigure}
\begin{subfigure}{0.49\textwidth}
\includegraphics[width=\textwidth]{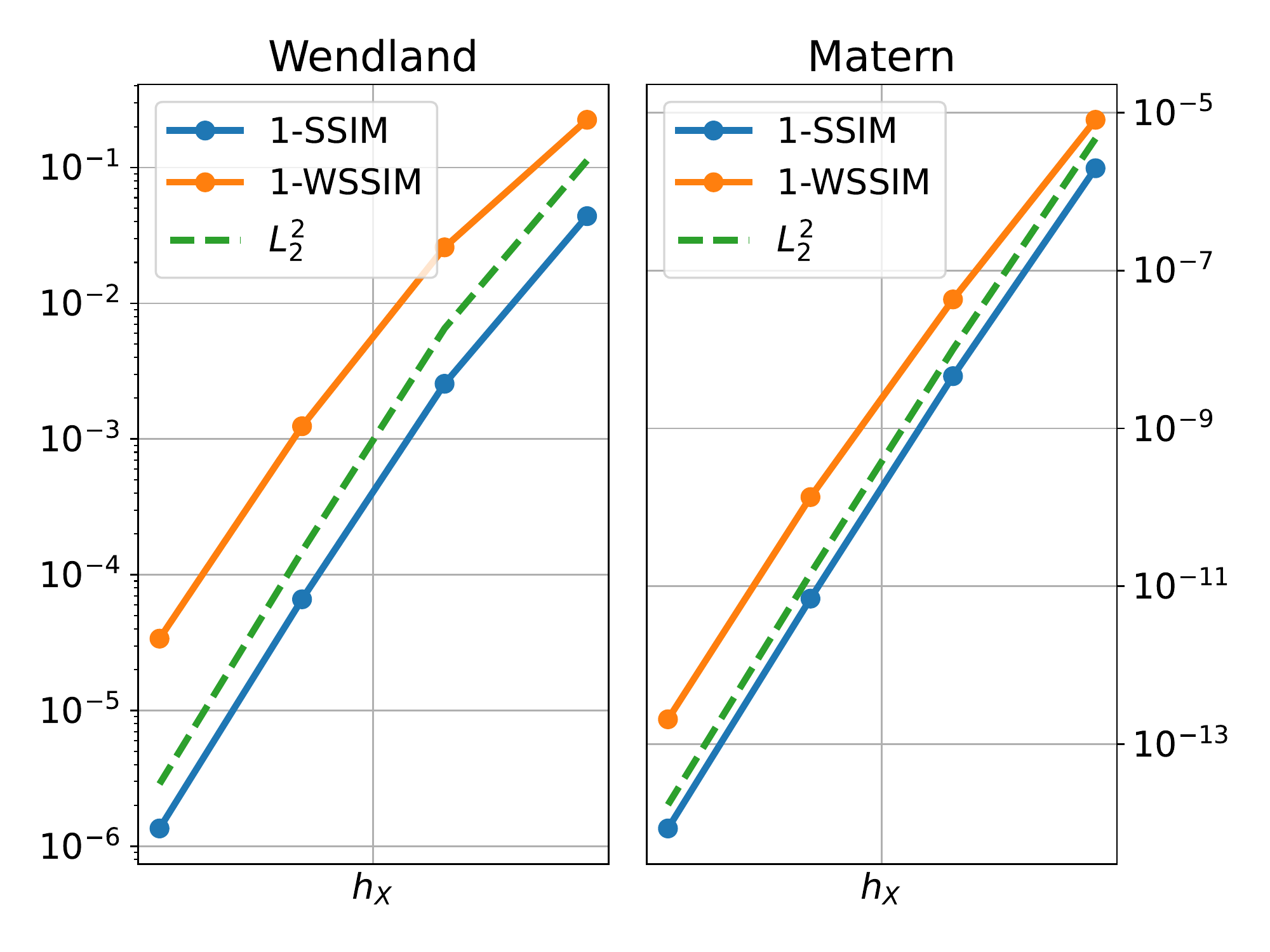}
\caption{}\label{fig:kernel_f2}
\end{subfigure}
\end{center}
\caption{Decay of the dissimilarity indices for bilinear and bicubic interpolation, with target function $f_1$ (Panel \eqref{fig:bicunear_f1}) and $f_2$ (Panel 
\eqref{fig:bicunear_f2}), and kernel interpolation with kernels $\varphi_1,\;\varphi_2$, with target function $f_1$ (Panel \eqref{fig:kernel_f1}) and $f_2$ 
(Panel \eqref{fig:kernel_f2}). The bilinar interpolant of $f_1$ (left figure in Panel \ref{fig:bicunear_f1}) is computed also for grid sizes $s_i$ with 
$i=5,\dots, 9$.}
\label{fig:funct_interp}
\end{figure}

Additionally, we report various constants that are estimated from the results of these interpolation processes for $f_1$ and $f_2$ 
(Table~\ref{tab:funct_interp}). Namely, for each interpolation method we estimate the constants $c_f$ and $c_{fg}$ of Theorem \ref{thm:cssim_upper_bound}, 
$C_f$, $C_{fg}$ of Corollary \ref{cor:wcssim_upper_bound}, and the minimal constants $\bar c$ and $\bar C$ so that the corresponding dissimilarity index is 
smaller than this constant times the $L_2$ error. Here $f$ is the target function and $g$ is the interpolant. All these values are found numerically for each 
interpolation set $X_i$, and the average value over $i=1, \dots, 4$ is reported in the table.
Observe that $c_f$ and $C_f$ are independent of the approximation method, and thus they are constant for each $f$. In all cases, $c_f$ is roughly twice as large 
as $c_{fg}$, as is the case for $C_f$ and $C_{fg}$ for $f_2$. In turn, the local constants $C_f$, $C_{fg}$ are larger than the global ones by two orders of 
magnitude for $f_2$, and $4$ to $6$ orders for $f_1$. 
Comparing these values with the optimal constants $\bar c, \bar C$, it seems that our estimates are quite sharp for the global index (Theorem 
\ref{thm:cssim_upper_bound}), since the effectivity ratio $c_f/\bar c$ is roughly of order $10$. 
In the case of the local estimate (Corollary \ref{cor:wcssim_upper_bound}) instead, the effectivity is larger than $10^2$ for $f_2$ and even $10^5$ for $f_1$. 
These results indicate that there is a quite large room for improvement of the constants in our asymptotic estimates, at least in the local case.

Additionally, we report in Table \ref{tab:funct_interp} the values $\bar r$, $\bar R$ of the estimated rates of decay of the two dissimilarity indices, i.e., so 
that $1-\mathrm{SSIM}(f, g)\leq c' h^{\tilde r}$ for some $c'>0$, and similarly for $\wssim$ and $\bar R$. These values are computed numerically by linear 
regression of the logarithms of the computed rates, and they all confirm the theoretical predicted rates of Section \ref{sec:int_methods}. Moreover, in the case 
of the polynomial methods these experimental rates are significantly faster than the expected ones, and this can be considered an instance of superconvergence 
which we argue to be due to the additional regularity of the test functions.

\begin{table}
\centering
\begin{tabular}{|l|l|cccc|cccc|}
\hline  
&&\multicolumn{4}{c|}{$\cssim$} & \multicolumn{4}{c|}{$\wcssim$}\\ 
&&    $c_f$ & $c_{fg}$ & $\bar c$ & $\bar r$ &    $C_f$ & $C_{fg}$ & $\bar C$ & $\bar R$ \\
\hline
$f_1$ &Bil. & 8.0e+01 &  4.2e+01 & 2.8e+01 & 4.1e+00 & 3.1e+07 &  7.4e+06 & 1.4e+03 & 3.4e+00 \\
&Bic.  & 8.0e+01 &  4.0e+01 & 5.6e+00 & 8.3e+00 & 3.1e+07 &  1.6e+07 & 4.1e+02 & 8.1e+00 \\
\cline{2-10}
&$\varphi_1$& 8.0e+01 &  4.0e+01 & 2.1e+00 & 4.5e+00 & 3.1e+07 &  6.9e+06 & 4.4e+01 & 4.9e+00 \\
&$\varphi_2$ & 8.0e+01 &  4.0e+01 & 3.7e+00 & 8.1e+00 & 3.1e+07 &  1.3e+07 & 1.1e+03 & 1.1e+01 \\
\hline
$f_2$&Bil. & 4.0e+00 &  2.0e+00 & 1.5e-01 & 4.0e+00 & 3.4e+02 &  1.7e+02 & 1.3e+00 & 3.8e+00 \\
&Bic.  & 4.0e+00 &  2.0e+00 & 5.1e-01 & 8.1e+00 & 3.4e+02 &  1.7e+02 & 2.0e+00 & 8.0e+00 \\
\cline{2-10}
&$\varphi_1$ & 4.0e+00 &  1.9e+00 & 4.2e-01 & 5.0e+00 & 3.4e+02 &  2.0e+02 & 6.5e+00 & 4.8e+00 \\
&$\varphi_2$ & 4.0e+00 &  2.0e+00 & 4.7e-01 & 9.3e+00 & 3.4e+02 &  1.7e+02 & 6.9e+00 & 8.8e+00 \\
\hline
\end{tabular}
\caption{Constants relating the decay of the SSIM dissimilarity ($c_f$, $c_{fg}$, $\bar c$) and the $\wssim$ dissimilarity ($C_f$, $C_{fg}$, $\bar C$) to the 
$L_2$ error, and computed rates of decay $\bar r, \bar R$, for $f_1$ and $f_2$.}
\label{tab:funct_interp}
\end{table}

\subsection{Image interpolation experiments}\label{sec:img_interp}

We now test bilinear and bicubic image interpolation on actual images, using the implementation of the two methods provided by the OpenCV Python library 
\cite{OpenCV}.
In particular, the derivatives used in the bicubic interpolants are estimated by taking $4\times 4$ neighborhoods in the images (see Section \ref{sec:hermite}).
We consider four $256\times 256$ images displayed in Figure \ref{fig:images} and whose values are normalized in $[0,1]$, which are well-known in the context of 
image processing. Each image is then undersampled to sizes $40\times40$, $80\times 80$, $160\times 160$, $320\times320$. The resized images are then 
interpolated in order to recover a $256\times 256$ image, and the reconstruction is compared to the original image. We point out that the spatial step of the 
interpolation dataset is defined to be the reciprocal of the number of pixels for each dimension, and that we use the same weight function as before to compute 
the $\wssim$.

The decay of the errors are reported in Figure \ref{fig:image_interp}, and also in this case there is an almost perfect agreement between the decay rates of the 
dissimilarity indices and of the squared $L_2$ norm. The corresponding constants, estimated as in the previous section, are reported in Table 
\ref{tab:image_interp}. Also in this case, similar considerations as before can be made regarding the ratios between the different constants, especially the 
fact that the theoretical bounds of Theorem \ref{thm:cssim_upper_bound} is roughly one order away from being optimal.

The rates of convergence of Section \ref{sec:int_methods} can not be applied because of the approximation of the derivatives and the irregularity of the 
functions which underlie the images. 
Nevertheless, the experimental rates $\bar r$ and $\bar R$ reported in Table \ref{tab:image_interp} show that for all images, and for both the local and global 
dissimilarity index, the bilinear interpolants converge with a rate between $1.1$ (\textit{baboon}) and $1.7$ (\textit{peppers}), and the bicubic ones between 
$1.4$ (\textit{baboon}) and $2.2$ (\textit{peppers}), with \textit{cameraman} and \textit{Lenna} in between.
This variability is probably due the presence of more complex structures and more frequent gray-value variations in the some of the images.

\begin{figure}
$
\begin{array}{c}
\includegraphics[width=0.23\textwidth]{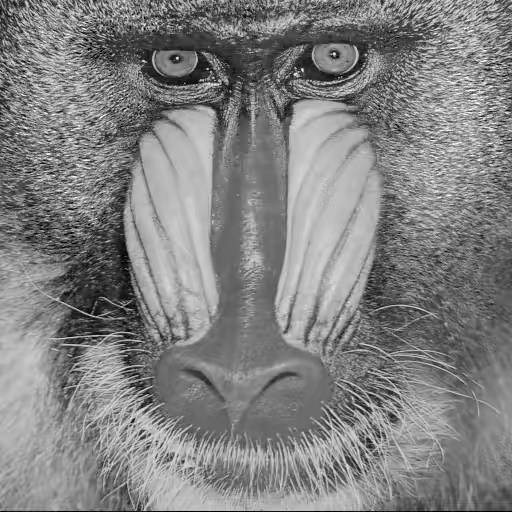} \hskip 5 pt
\includegraphics[width=0.23\textwidth]{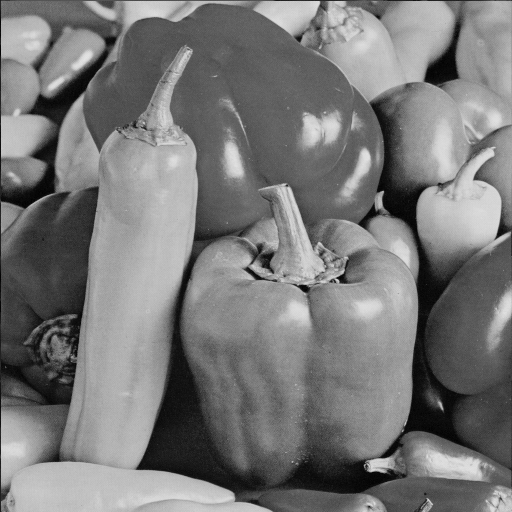}\hskip 5 pt
\includegraphics[width=0.23\textwidth]{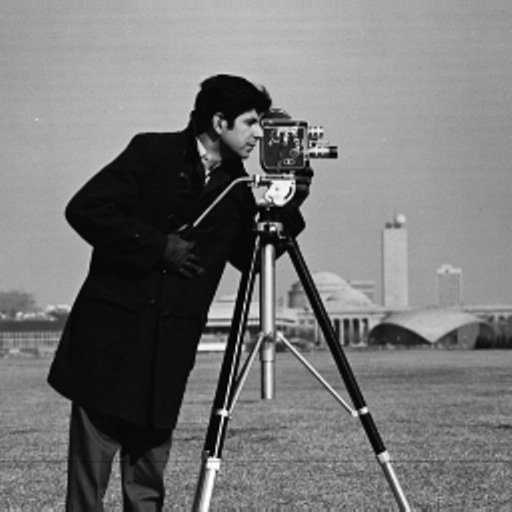} \hskip 5 pt
\includegraphics[width=0.23\textwidth]{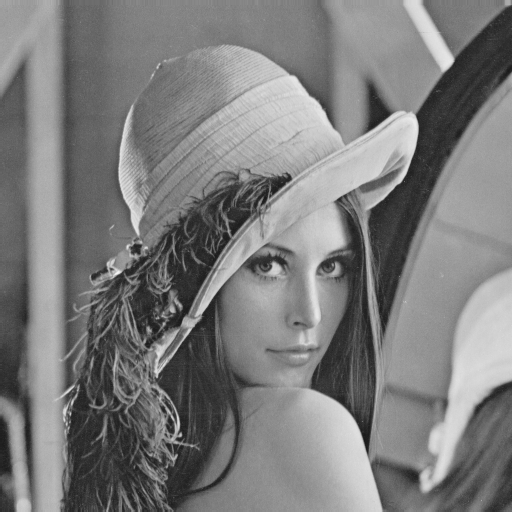}
\end{array}
$
\caption{Test images used in Section \ref{sec:img_interp}: from left to right \textit{baboon},  \textit{peppers}, \textit{cameraman}, \textit{Lenna}.}
\label{fig:images}
\end{figure}

\begin{figure}
\begin{center}
\begin{subfigure}{0.49\textwidth}
\includegraphics[width=\textwidth]{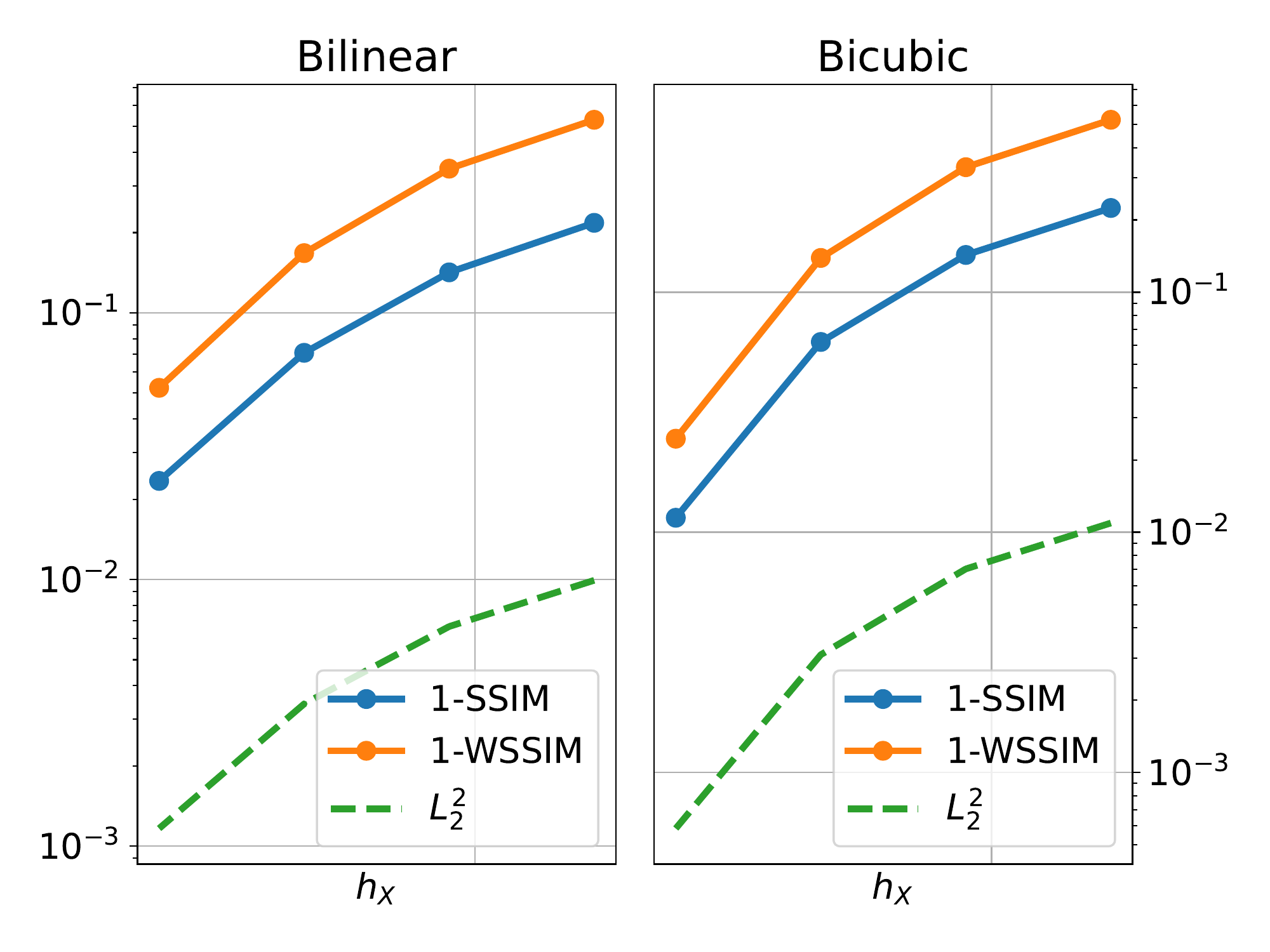}
\caption{}\label{fig:bab}
\end{subfigure}
\begin{subfigure}{0.49\textwidth}
\includegraphics[width=\textwidth]{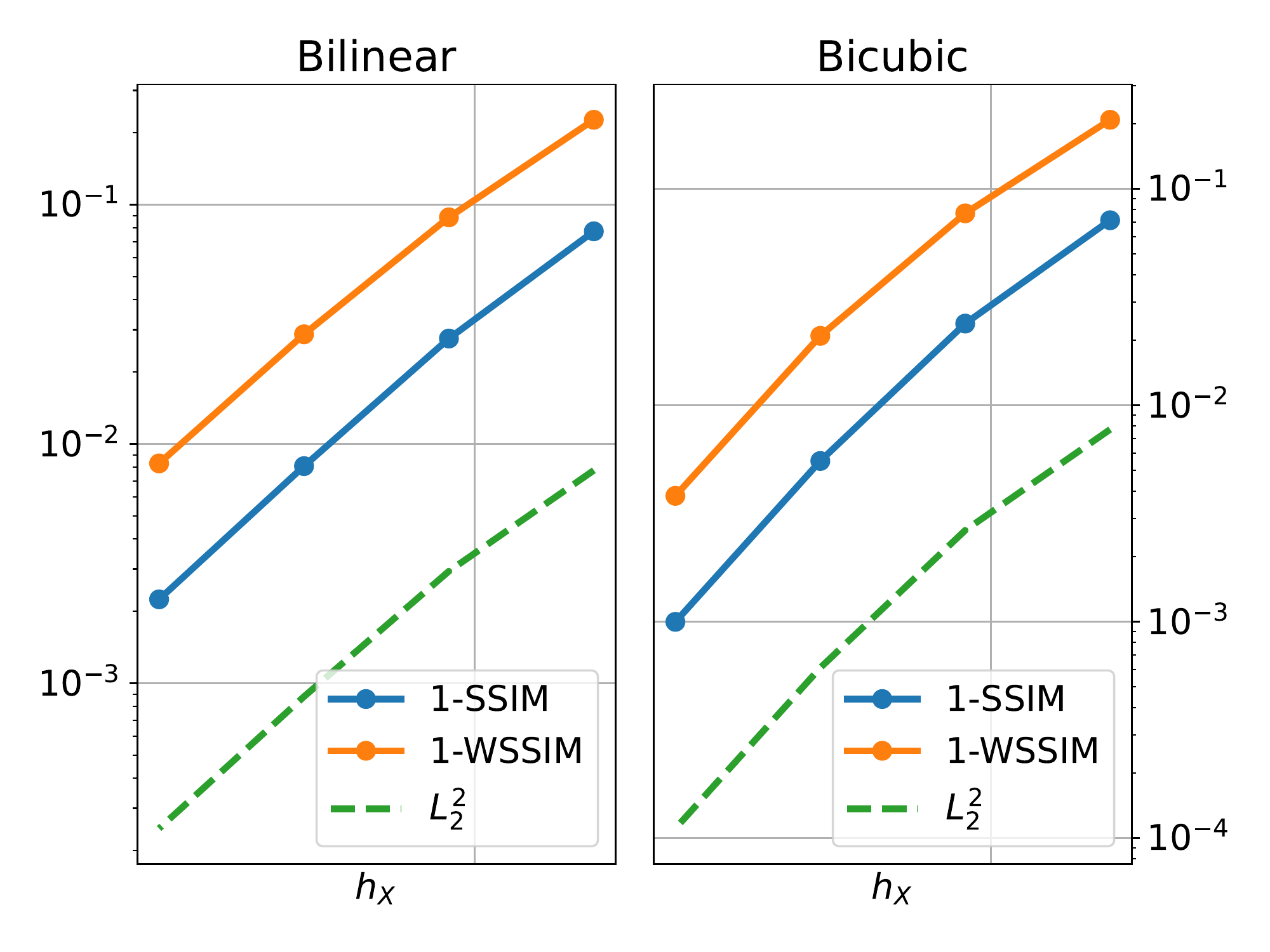}
\caption{}\label{fig:peppers}
\end{subfigure}
\begin{subfigure}{0.49\textwidth}
\includegraphics[width=\textwidth]{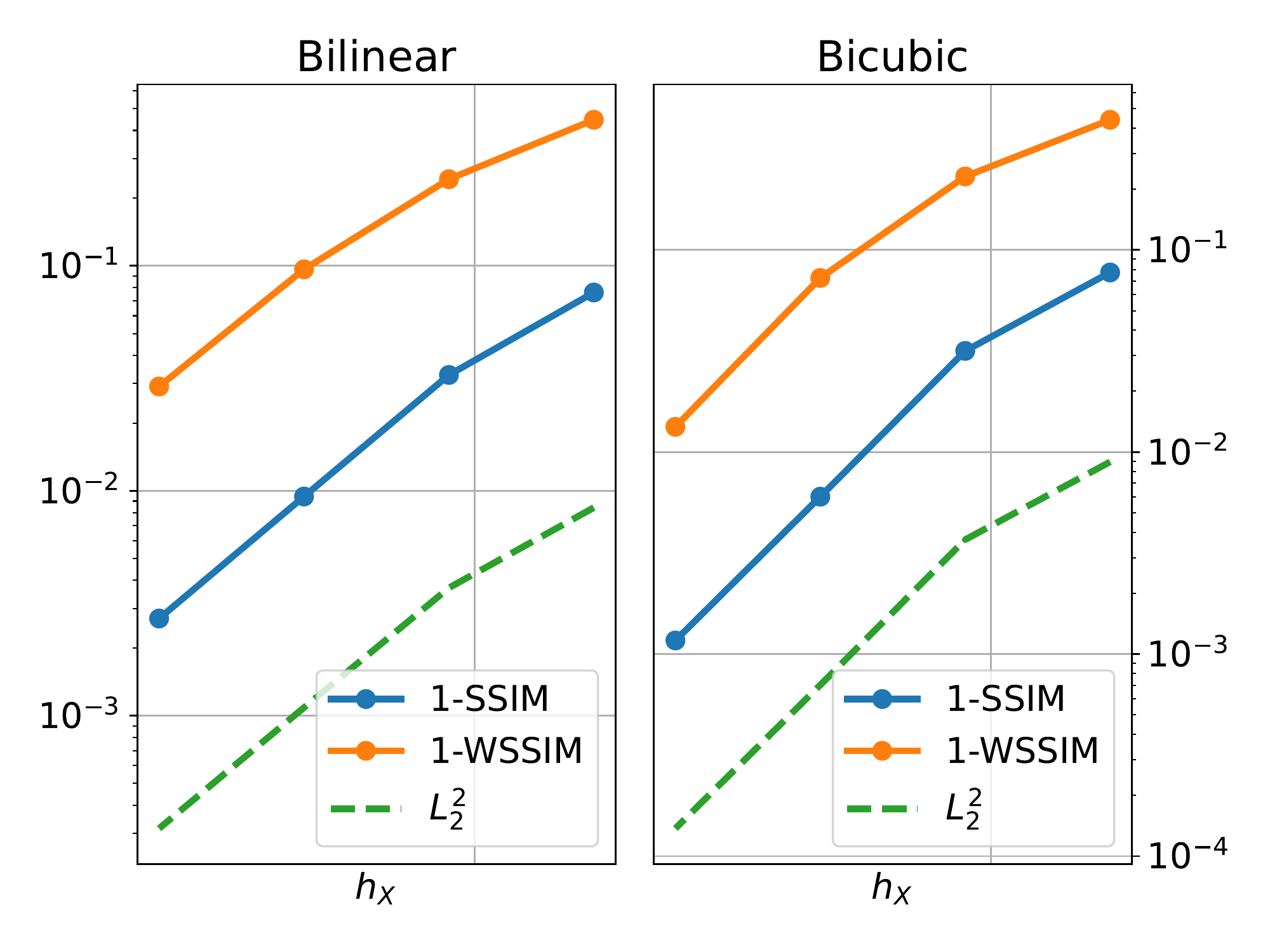}
\caption{}\label{fig:cam}
\end{subfigure}
\begin{subfigure}{0.49\textwidth}
\includegraphics[width=\textwidth]{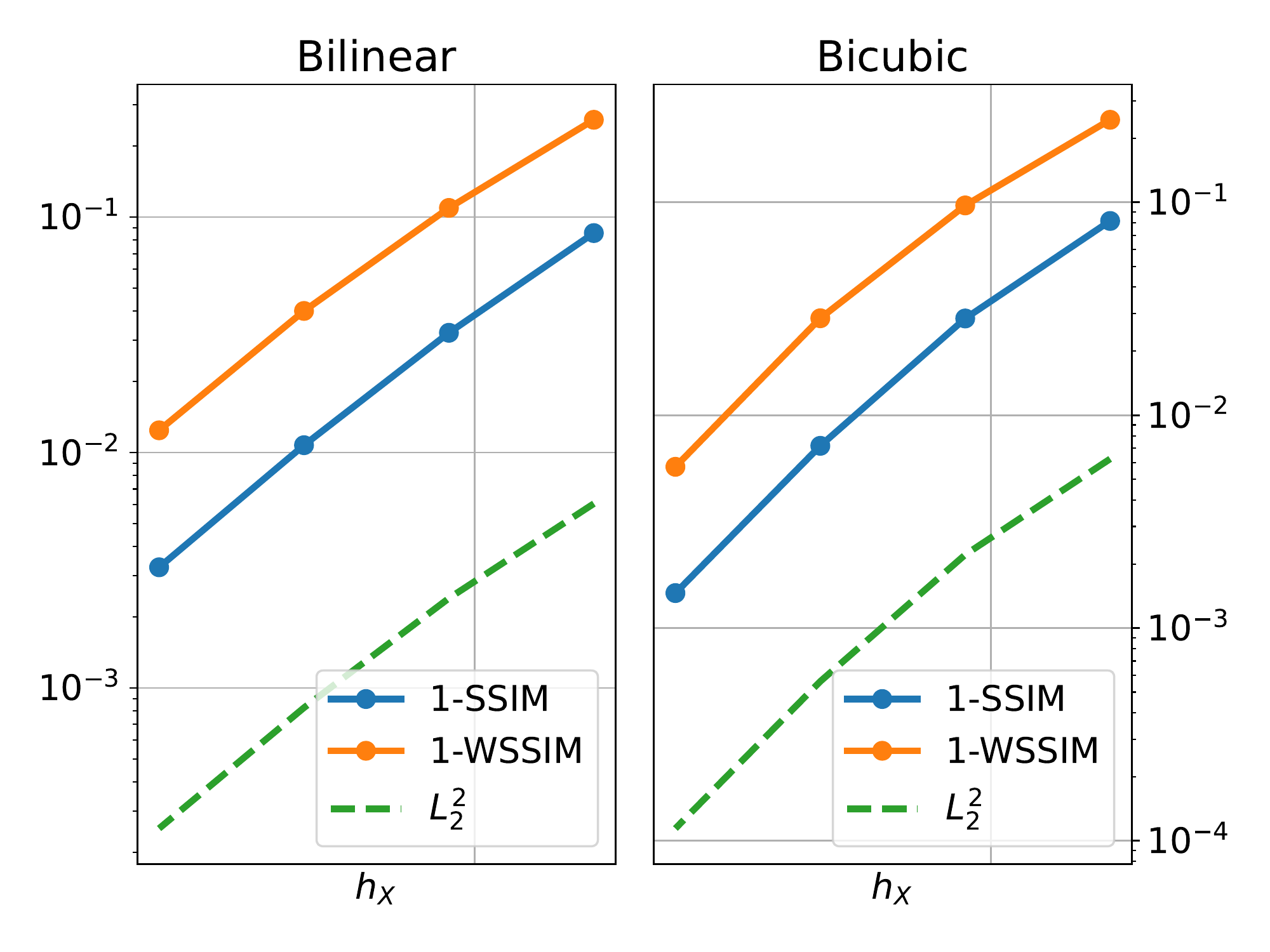}
\caption{}\label{fig:lenna}
\end{subfigure}
\end{center}

\caption{Decay of the dissimilarity indices for bilinear and bicubic interpolation of \textit{baboon} (Panel \eqref{fig:bab}), \textit{peppers} (Panel \eqref{fig:peppers}), \textit{cameraman} (Panel \eqref{fig:cam}), and \textit{Lenna} (Panel \eqref{fig:lenna}).}
\label{fig:image_interp}
\end{figure}

\begin{table}
\centering
\begin{tabular}{|l|l|cccc|cccc|}
\hline  
&&\multicolumn{4}{c|}{$\cssim$} & \multicolumn{4}{c|}{$\wcssim$}\\ 
&&    $c_f$ & $c_{fg}$ & $\bar c$ & $\bar r$ &    $C_f$ & $C_{fg}$ & $\bar C$ & $\bar R$ \\
\hline
$I_1$&Bil. & 1.5e+02 &  8.6e+01 & 2.1e+01 & 1.1e+00 & 8.2e+03 &  5.7e+03 & 5.0e+01 & 1.4e+00 \\
&Bic. & 1.5e+02 &  8.2e+01 & 2.0e+01 & 1.4e+00 & 8.2e+03 &  5.2e+03 & 4.5e+01 & 1.9e+00 \\
\hline
$I_2$&Bil. & 7.5e+01 &  3.9e+01 & 9.4e+00 & 1.7e+00 & 4.9e+04 &  3.0e+04 & 3.1e+01 & 1.7e+00 \\
&Bic. & 7.5e+01 &  3.8e+01 & 9.1e+00 & 2.1e+00 & 4.9e+04 &  2.8e+04 & 3.1e+01 & 2.2e+00 \\
\hline
$I_3$&Bil. & 7.2e+01 &  3.7e+01 & 8.8e+00 & 1.6e+00 & 1.7e+05 &  1.1e+05 & 7.5e+01 & 1.5e+00 \\
&Bic.  & 7.2e+01 &  3.7e+01 & 8.6e+00 & 2.1e+00 & 1.7e+05 &  9.6e+04 & 7.8e+01 & 2.1e+00 \\
\hline
$I_4$&Bil. & 1.0e+02 &  5.5e+01 & 1.3e+01 & 1.6e+00 & 9.6e+04 &  6.7e+04 & 4.6e+01 & 1.6e+00 \\
&Bic. & 1.0e+02 &  5.3e+01 & 1.3e+01 & 1.9e+00 & 9.6e+04 &  6.2e+04 & 4.6e+01 & 2.0e+00 \\
\hline
\end{tabular}
\caption{Constants relating the decay of the SSIM dissimilarity ($c_f$, $c_{fg}$, $\bar c$) and the the $\wssim$ dissimilarity ($C_f$, $C_{fg}$, $\bar C$) to the $L_2$ error, and computed rates of decay $\bar r, \bar R$, for the interpolation of the images
$I_1$ (\textit{baboon}),
$I_2$ (\textit{peppers}),
$I_3$ (\textit{cameraman}), 
$I_4$ (\textit{Lenna}).}
\label{tab:image_interp}
\end{table}

\section{Discussion and conclusions}\label{sec:conclusions}
In this paper we considered the continuous SSIM introduced in \cite{Marchetti2021a} and extended it to a weighted and local version, and we shed some light on their relation with the classical SSIM and with the $L_2$ norm. In particular, we deepened the relationship between cSSIM and $L_2$ norm by providing both upper and lower bounds for the cSSIM (Section \ref{sec:cssim_l2}), and we used those bounds to present a detailed analysis of the cSSIM-convergence rates of some well-known image interpolation methods. To our knowledge, these are the first explicit results on convergence rates of image interpolation methods in terms of the SSIM. 
The numerical tests carried out in Section \ref{sec:num_tests} confirm the theoretical findings of Section \ref{sec:int_methods}, as well as the statement of Theorem \ref{thm:cssim_upper_bound}. 

These theoretical results and experimental findings proved that one may infer significant information on the behavior of the SSIM by looking at the much more 
classical $L_2$ error, which is a convex, thoroughly studied, and easy to implement measure. On the other hand, we discussed how the lower bounds of Section 
\ref{sec:equivalence} may be used to improve the understanding of why the SSIM is found to be often superior to the $L_2$ norm.

Future work will be devoted to the investigation of the insight that complex structures may lead to different rates of approximation. Exploiting such structures in the images may lead to further and more accurate theoretical bounds in terms of the cSSIM. In particular, extensive testing on super resolution \cite{Yang2014} and image interpolation \cite{Yeganeh2015} benchmarks may further quantify the effectivity of the new bounds.
Moreover, the SSIM has been used as a loss function in supervised learning to enforce a perceptually accurate reconstruction of images. In this setting, some approaches have been adopted to overcome the fact that the SSIM (as well as the cSSIM) induces a non convex loss \cite{Brunet12a}. In view of our theoretical results, it seems interesting to investigate if the minimization of the $L_2$ loss may provide a reliable surrogate of the SSIM loss, or at least an accurate initialization.

\section*{Acknowledgments}
The authors acknowledge their membership in the ``Research Italian Network on Approximation (RITA)'', in the GNCS-IN$\delta$AM and in the group ``Teoria dell’Approssimazione e Applicazioni (TAA)" of the Unione Matematica Italiana (UMI).

\bibliographystyle{siamplain}
\bibliography{references}

\end{document}